\documentclass[preprint,3p,times]{elsarticle}

\usepackage{etoolbox}
\usepackage{graphicx} 
\usepackage{amsthm}
\newtheorem{lemma}{Lemma}
\newtheorem{theorem}{Theorem}
\newtheorem{remark}{Remark}
\newtheorem{corollary}{Corollary} 
\usepackage{amssymb}
\usepackage{amsmath}
\usepackage{comment}
\usepackage{relsize}
 
\usepackage[hidelinks]{hyperref}
\usepackage{lineno}
\usepackage{subcaption}
\usepackage{caption}
\usepackage{overpic}
\usepackage{array}
\usepackage{xcolor}
\usepackage{mathtools}
\usepackage{stmaryrd}
\usepackage{booktabs}
\usepackage[numbers]{natbib}

\newcommand{\avg}[1]{\left\{\hspace*{-1pt}\left\{#1\right\}\hspace*{-1pt}\right\}}

\newcommand{\jump}[1]{\ensuremath{\left\llbracket #1 \right\rrbracket}}

\newcommand{\jumpbig}[1]{\ensuremath{\big\llbracket #1 \big\rrbracket}}

\newcommand{\bB}[1]{\boldsymbol{#1}}
\newcommand{\mbf}[1]{\mathbf{#1}}

\newcolumntype{P}[1]{>{\centering\arraybackslash}m{#1}}

\makeatother

\usepackage{cancel}
\usepackage{orcidlink}

\begin{document}

\begin{frontmatter}

\title{Entropy analysis and entropy stable DG methods for the 1D shallow water moment equations}

\author[UoG]{Julio Careaga\,\orcidlink{0000-0003-4466-6173}}
\author[liu]{Patrick Ersing\corref{cor1}\,\orcidlink{0009-0005-3804-5380}}
\ead{patrick.ersing@liu.se}
\cortext[cor1]{Corresponding author}
\author[UoG,GU]{Julian Koellermeier\orcidlink{0000-0002-8822-461X}}
\author[liu]{Andrew R. Winters\orcidlink{0000-0002-5902-1522}}

\affiliation[UoG]{organization={Bernoulli Institute, University of Groningen},
				addressline={Nijenborgh 9},
				city={Groningen},
				postcode={9747},
				country={The Netherlands}}
\affiliation[liu]{organization={Department of Mathematics, Linköping University},
	city={Linköping},
	postcode={58183}, 
	country={Sweden}}
\affiliation[GU]{organization={Department of Mathematics, Computer Science and Statistics, Ghent University},
	addressline={Krijgslaan 299 - S9}, 
	city={Gent},
	postcode={9000}, 
	country={Belgium}}

\begin{abstract}
\noindent We demonstrate that the one-dimensional shallow water moment equations satisfy an auxiliary entropy conservation law, where the entropy function corresponds to the total energy. Additionally, we show that the classical Newtonian slip friction and Manning friction terms are entropy dissipative with respect to the developed entropy variables.
The results from the continuous entropy analysis are used to construct an entropy stable and well-balanced nodal discontinuous Galerkin spectral element method for the spatial approximation.
Key to ensure the entropy stability of the scheme is the derivation of entropy conservative numerical fluxes that satisfy a discrete version of the entropy flux compatibility condition.
Finally, numerical examples demonstrate the performance of the scheme and validate the theoretical results.
\end{abstract}


\end{frontmatter}

\section{Introduction}

Shallow water equations (SWE) describe the evolution of a fluid under the assumption that the water thickness is orders of magnitude smaller than the horizontal scale. From the conservation of mass and momentum balance of the fluid for a free-boundary problem, the one-dimensional SWE consist of a two-by-two system for the water thickness and depth-averaged velocity. Despite the fact that SWE make a significant reduction in the difficulty of solving the incompressible Navier-Stokes equations, the reduction of the velocity profile to one depth-averaged variable may be overly strong and therefore lead to inaccuracies \cite{Kern2010, Sanvitale2016}. Shallow water moment equations (SWME) arise as an extension of the SWE, and assume that the velocity profile is polynomial in the vertical coordinate \cite{Kowalski2019}.
In SWME, the moments correspond to the coefficients of the polynomial expansion, which become part of the unknowns of the system together with the water height and the depth-averaged velocity. Therefore, increasing the order of this polynomial expansion enhances the accuracy when modelling velocity profiles with highly nonlinear variations with respect to depth. For a zero-order polynomial expansion the SWME coincide with the standard SWE. Further variants of the SWME have been proposed in the literature, two of the most relevant ones are: the hyperbolic shallow water moment equations in \cite{Koellermeier2020c}, and the shallow water linearized moment equations (SWLME) in \cite{Koellermeier2020i}. These models are developed from the general SWME. The first discards some terms, in order to guarantee hyperbolicity and in the case of the second model, also the analytical computation of steady states. Other works have extended the SWME to non-hydrostatic flows \cite{Scholz2023}, and to two-dimensional horizontal domains \cite{Bauerle2025}, while a general framework that extends the SWME to granular flows over inclined planes is proposed in \cite{CareagaArxiv2025}.

In the standard SWE, the governing equations ensure conservation of mass, while the momentum balance is described by a depth-averaged velocity equation. To derive an energy equation, in \cite{Gassner2016}, the authors use a skew-symmetric formulation of the SWE to obtain both the kinetic and potential energy equations, and the total energy as a result of their sum.  Analogous to the SWE, the derived total energy corresponds to a convex mathematical entropy function for the SWME, which enables the construction of entropy stable schemes.
In conservation laws, the mathematical entropy is related to the existence of physically relevant solutions, and also to the nonlinear stability of numerical schemes \cite{Tadmor1984,Tadmor1987,Tadmor2003}. A methodology to derive entropy stable discontinuous Galerkin methods can be found in \cite{Gassner2021} and \cite{Winters2021}, where the authors use summation-by-parts properties and a flux-differencing formulation to derive entropy conservative fluxes. The numerical schemes developed in \cite{Ersing2024, Ersing2025} use this approach for multilayer shallow water models, including nonconservative terms and ensuring the well-balancing property. In the same line, an entropy stable discontinuous Galerkin method in fluctuation form applied for the Saint-Venant-Exner system, is developed in \cite{ersing2026new}. In the context of SWLME, in \cite{FanArxiv2025}, the authors develop a path-conservative and well-balanced discontinuous Galerkin scheme by introducing equilibrium variables.

As with SWE, the SWME satisfy the conservation of mass, while the momentum balance can be described from the equations for the depth-averaged velocity and moments. However, due to the nonlinear coupling between the mean velocity equation and moment equations, the derivation of a kinetic energy equation becomes non-trivial. 
Building on the procedure presented in \cite{Gassner2016} for the SWE, we derive an energy equation for the SWME. 

The first aim of this work is to derive this energy equation for the SWME in the general form introduced in \cite{Kowalski2019}, where the total kinetic energy accounts for contributions from the moments.
Unlike previous contributions made for the SWLME \cite{FanArxiv2025, KoellermeierEnergy2025arxiv, Caballero2025}, we deal here with the nonlinear terms arising from the coupling between the average velocity and moments. 
Additionally, we demonstrate that the inclusion of two common friction terms has a dissipative contribution to the energy equation.
The second aim is to develop an entropy stable discontinuous Galerkin numerical scheme that can handle the non-conservative products, a common feature of shallow water models, as well as steep gradients due to possibly discontinuous solutions. 

This paper is organized as follows. In Section~\ref{sec:SWME}, the kinetic, potential and total energy equations are derived for the one-dimensional SWME. In addition, a convex entropy function, entropy flux, and entropy variables are established as a consequence of the total energy formulation. In Section~\ref{sec:Entropy:dissipation}, we show that the Newtonian slip and Newtonian Manning friction terms are, in fact, entropy dissipative. The numerical approximation is described in Section~\ref{sec:numerical:scheme}. The discontinuous Galerkin scheme and corresponding conditions for semi-discrete entropy preservation are discussed in Section~\ref{sec:DG:description} and the construction of the entropy conservative numerical flux is treated in Section~\ref{sec:EC-Moment-Fluxes}. In Section~\ref{sec:es_dgsem} we then construct an entropy stable scheme by adding suitable numerical dissipation at element interfaces.
A series of illustrative numerical examples showing the performance and reliability of the developed numerical scheme are presented in Section~\ref{sec:numerical:examples}. Conclusions and final remarks are included in Section~\ref{sec:conclusions}.


\section{Continuous entropy analysis for the SWME} \label{sec:SWME}

We consider the one-dimensional shallow water moment equations (SWME), originally introduced in \cite{Kowalski2019}, for the case of frictionless conditions and including a variable bathymetry. Given the spatial coordinate $x\in \Omega:=(x_{\rm a},x_{\rm b})$ with $x_{\rm a}<x_{\rm b}$, we assume that the bathymetry function is defined by the time-independent function $b=b(x)$. The sought unknowns in the shallow water moment model are the water height $h=h(x,t)$, the depth-averaged or mean velocity $u_m = u_m(x,t)$, and the $N$ moments $\alpha_i=\alpha_i(x,t)$ for all $x\in\Omega$ and $t>0$. The mean velocity together with the moments determine the horizontal component of the fluid velocity, denoted by $u$, through the following ansatz:
\begin{align}\label{eq:velocity:ansatz}
 u(x,\zeta,t) = u_m(x,t) + \sum_{i=1}^N\alpha_i(x,t)\phi_i(\zeta),
\end{align}
where $\zeta= (z-b)/h \in[0,1]$ is the scaled and normalized vertical coordinate, $z$ is the original vertical coordinate, and $\phi_1,\phi_2,\dots,\phi_N$ are the first $N$ shifted Legendre polynomials defined in the unit interval and normalized setting $\phi_i(0)=1$ for all $i=1,\dots,N$.
The polynomials $\phi_i$ are obtained from the classical Legendre polynomials $L_i=L_i(\zeta)$ defined in the interval $[-1,1]$, by
\begin{align}\label{eq:shifted_legendre_polynomial}
 \phi_i(\zeta) \coloneqq (-1)^i L_i(2\zeta-1),\qquad \text{ for} \quad i=1,\dots,N,\quad 0\leq \zeta\leq 1.
\end{align}
The first three shifted Legendre polynomials are
\begin{equation*}
	\phi_1(\zeta) = 1-2\zeta, \quad 
	\phi_2(\zeta) = 6\zeta^2 - 6\zeta + 1, \quad 
	\phi_3(\zeta) = -20\zeta^3 + 30\zeta^2 - 12\zeta+1.
\end{equation*}
In turn, the SWME model is built from the assumption that the velocity of the fluid is polynomial on the vertical coordinate. The SWME in \cite{Kowalski2019} describes the evolution of $h$, $u_m$ and $\alpha_i$ for $i=1,\dots,N$, and is formulated as follows:
\begin{subequations} \label{sys:SWME}
\begin{align}
\partial_t h+\partial_x\left(h u_m\right) &= 0, \label{eq:SWME_C}\\
\partial_t\!\left(h u_m\right)+\partial_x \Big(h u_m^2+ \Psi +\tfrac{1}{2} g h^2\Big) &=  - gh \partial_x b, \label{eq:SWME_M1}\\
\partial_t\!\left(h \alpha_i\right)+\partial_x\Big( 2 h u_m \alpha_i+\mathfrak{A}_i \Big) &=  u_m \partial_x\left(h \alpha_i\right) - \mathfrak{B}_i,\label{eq:SWME_alphai}
\end{align}
\end{subequations}
for all $x\in \Omega$ and $t>0$, where $\Psi$ in \eqref{eq:SWME_M1} is
\begin{align} \label{eq:SWLME_defPsi}
 \Psi \coloneqq h \sum_{i=1}^N\dfrac{1}{2i+1}\alpha_i^2 = h\Big(\tfrac{1}{3}\alpha_1^2 + \tfrac{1}{5}\alpha_2^2 + \dots + \tfrac{1}{2N+1}\alpha_N^2\Big).
\end{align}
The terms $\mathfrak{A}_i$ and $\mathfrak{B}_i$ in \eqref{eq:SWME_alphai} introduce additional nonlinearities to the system as they involve products of the moments and their partial derivatives, as well as factors of $h$. They are given by
\begin{align*}
\mathfrak{A}_i\coloneqq  h\sum_{j, k=1}^N A_{i j k} \alpha_j \alpha_k,\qquad \mathfrak{B}_i\coloneqq \sum_{j, k=1}^N \!B_{i j k} \alpha_k \partial_x\big(h \alpha_j\big),
\end{align*}
where the following coefficients are used
\begin{equation}\label{eq:Aijk_Bijk}
    A_{i j k} =(2i+1) \int_0^1 \phi_i(\zeta) \phi_j(\zeta) \phi_k(\zeta)\, {\rm d} \zeta, \qquad
    B_{i j k} = (2 i+1) \int_0^1 \phi_i^{\prime}(\zeta)\left(\int_0^\zeta \phi_j(s){\rm d} s\right) \phi_k(\zeta)\, {\rm d} \zeta.
\end{equation}
Note that $\mathfrak{B}_i$ in \eqref{eq:SWME_alphai} corresponds to a non-conservative term.
Defining the $N$-dimensional vectors $\bB{\alpha} := (\alpha_1,\alpha_2,\dots,\alpha_N)^{\rm T}$, $\bB{\mathfrak{A}} := (\mathfrak{A}_1,\mathfrak{A}_2,\dots,\mathfrak{A}_N)^{\rm T}$ and $\bB{\mathfrak{B}} := (\mathfrak{B}_1,\mathfrak{B}_2,\dots,\mathfrak{B}_N)^{\rm T}$, we can conveniently write \eqref{eq:SWME_alphai} as
\begin{align*}
 \partial_t\!\left(h \bB{\alpha}\right)+\partial_x\big( 2 h u_m \bB{\alpha}+\bB{\mathfrak{A}} \big) &=  u_m \partial_x\left(h\bB{\alpha}\right) -  \bB{\mathfrak{B}}.
\end{align*}
A variant of the SWME, the so-called shallow water linearized moment equations (SWLME), is developed in \cite{Koellermeier2020i}. 
Therein, the model takes $\bB{\mathfrak{A}} = \bB{\mathfrak{B}} = \bB{0}$ so that
\begin{subequations} \label{sys:SWLME}
\begin{align}
\partial_t h+\partial_x\left(h u_m\right) &= 0, \label{eq:SWLME_C}\\
\partial_t\!\left(h u_m\right)+\partial_x \Big(h u_m^2+ \Psi +\tfrac{1}{2} g h^2\Big) &=  - gh \partial_x b, \label{eq:SWLME_M1}\\
\partial_t\!\left(h \bB{\alpha}\right)+\partial_x( 2 h u_m \bB{\alpha}) &=  u_m \partial_x\left(h\bB{\alpha}\right).
\label{eq:SWLME_alphai}
\end{align}
\end{subequations}

Throughout the remainder of this section, we assume positivity of $h$ and that both the solution and bottom topography are sufficiently smooth. In what follows, we derive energy equations associated with~\eqref{sys:SWME}. This is based on the hierarchical structure of the SWME: in the zeroth-order case, where all moments vanish, the system reduces to the SWE. Consequently, the energy equation derivation developed for the SWE in~\cite{Gassner2016} can be directly applied to the system~\eqref{sys:SWME}.
The linearized model SWLME system \eqref{sys:SWLME} is hyperbolic and its eigenvalues are given by \cite[Theorem 1]{Koellermeier2020i}:
\begin{align} \label{eq:eigenvalues:SWLME}
 \lambda_{1,2} = u_m \pm \bigg( gh + \sum_{i=1}^N \dfrac{3\alpha_i^2}{2i+1}\bigg)^{1/2},\qquad \lambda_{i+2} = u_m,\qquad\text{for }i=1,\dots,N.
\end{align}
Unlike the SWLME, the general SWME system \eqref{sys:SWME} is only conditionally hyperbolic, and for certain values of the unknowns the hyperbolicity is lost \cite{Koellermeier2020c}.

In the next lemma, we provide an averaged kinetic energy equation for the mean velocity and a potential energy equation for system~\eqref{sys:SWME}.
\begin{lemma}\label{lem:partial:energy}
Let $h$, $u_m$ and $\bB{\alpha}$ be a smooth solution of the SWME~\eqref{sys:SWME}. Then, the following averaged kinetic and potential energy equations hold
\begin{align}
 \partial_t\left(\frac{h u_m^2}{2}\right) + \partial_x\left(\frac{h u_m^3}{2}\right) +g h u_m \partial_x(h+b)+ u_m \partial_x \Psi& =0, \label{eq:SWME_K}\\
\partial_t\left(g \frac{h^2}{2}+g h b\right)+g(h+b) \partial_x(h u_m)& =0, \label{eq:SWME_P}
\end{align}
where $\tfrac{1}{2}h u_m^2$ is the kinetic energy in SWE, and $\tfrac{1}{2}gh^2 + ghb$ is the potential energy.
\end{lemma}

\begin{proof}
To derive \eqref{eq:SWME_K}, we follow a similar approach as in the derivation of the kinetic energy equation in \cite{Gassner2016}. Subtracting~\eqref{eq:SWME_C} multiplied by $u_m$ from the momentum equation~\eqref{eq:SWME_M1}, we get
\begin{align}
    h \partial_t u_m +  hu_m \partial_x u_m +gh \partial_x (h+b) + \partial_x  \Psi&= 0. \label{eq:SWME_A}
\end{align}
Next, the arithmetic average between \eqref{eq:SWME_M1} and \eqref{eq:SWME_A} gives
\begin{align}
    \frac{1}{2}\Big(\partial_t(h u_m)+h \partial_t u_m\Big) + \frac{1}{2}\Big(\partial_x\!\left(hu_m^2\right)+ h u_m \partial_x u_m\Big)+g h \partial_x(h+b) + \partial_x  \Psi&=0, \label{eq:SWLME_S}
\end{align}
and multiplying this equation by $u_m$ and further grouping terms we arrive at the desired averaged kinetic energy equation \eqref{eq:SWME_K}. As~\eqref{eq:SWME_C} coincides with the equation for conservation of mass in the SWE, the potential energy equation remains the same. Therefore, we omit its derivation and refer the reader to \cite{Gassner2016}.
\end{proof}

Now, before deriving an equation for the total kinetic energy, including the contributions by the $N$ moment coefficients, we must first analyze the term containing $\partial_x\Psi$ in the averaged kinetic energy equation \eqref{eq:SWME_K}.
In the following auxiliary lemma, we provide a useful equation for the term $\Psi$ from  \eqref{eq:SWLME_defPsi}.
\begin{lemma} \label{lem:Psi:equation}
Let $h$, $u_m$ and $\bB{\alpha}$ be a smooth solution of the SWME~\eqref{sys:SWME}. Then the auxiliary variable $\Psi$ defined by \eqref{eq:SWLME_defPsi}, which depends on $h$ and $\bB{\alpha}$, satisfies the following equation
 \begin{align} \label{eq:SWME_Psi}
 \partial_t\! \left(\dfrac{\Psi}{2}\right) + \partial_x \!\left(u_m\dfrac{\Psi}{2}\right) + \Psi \partial_x u_m + \sum_{i=1}^N \dfrac{1}{2i+1}\alpha_i\big(\partial_x \mathfrak{A}_i + \mathfrak{B}_i\big)= 0.
\end{align}
\end{lemma}

\begin{proof}
Subtracting \eqref{eq:SWME_C} multiplied with $\alpha_i$ from the moment equation \eqref{eq:SWME_alphai}, we obtain
\begin{align}
    h \partial_t \alpha_i + h \partial_x (u_m \alpha_i ) +\partial_x \mathfrak{A}_i + \mathfrak{B}_i &= 0, \qquad \text{for }i=1,\dots,N.
    \label{eq:SWME_Aalphai}
\end{align}
Next, the result from \eqref{eq:SWME_Aalphai} multiplied with $\alpha_i$ plus \eqref{eq:SWME_C} multiplied by $\tfrac{1}{2}\alpha_i^2$ yields the following auxiliary equations for each $i=1,\dots,N$
\begin{align}
\partial_t \left(h \frac{\alpha_i^2}{2}\right)  +  \partial_x \left( h u_m \frac{\alpha_i^2}{2}\right) + \alpha_i^2 h \partial_x u_m + \alpha_i \big(\partial_x \mathfrak{A}_i + \mathfrak{B}_i \big) &= 0, \label{eq:SWME_AUX}
\end{align}
where we have used the product rule and the identity
\begin{align*}
\alpha_i h \partial_x(u_m \alpha_i ) + \frac{\alpha_i^2}{2} \partial_x (h u_m) =\partial_x \left( h u_m \frac{\alpha_i^2}{2}\right) + \alpha_i^2 h \partial_x u_m.
\end{align*}
Then, dividing~\eqref{eq:SWME_AUX} by $2i+1$ and summing over the moments from $1$ to $N$, we obtain \eqref{eq:SWME_Psi}.
\end{proof}

Equation~\eqref{eq:SWME_Psi} from Lemma~\ref{lem:Psi:equation} can be seen as a partial kinetic energy equation. The total kinetic energy equation for the SWME can now be obtained by combining \eqref{eq:SWME_K} and \eqref{eq:SWME_Psi}. However, the last term on the left-hand side of \eqref{eq:SWME_Psi}, which involves $\mathfrak{A}_i$ and $\mathfrak{B}_i$, is not formulated in conservative form. Denoting this term by $\mathcal{Q}$, we have
\begin{align}\label{eq:coefs}
 \begin{aligned}
 \mathcal{Q}  \coloneqq \,
 \sum_{i=1}^N \dfrac{1}{2i+1}\alpha_i\big(\partial_x \mathfrak{A}_i + \mathfrak{B}_i\big)
 = \, \sum\limits_{i,j,k=1}^N \Big(\big( \widetilde{A}_{i j k} + \widetilde{B}_{i j k}\big)\alpha_i \alpha_k  \partial_x (h \alpha_j) + \widetilde{A}_{i j k} h \alpha_i \alpha_j  \partial_x \alpha_k \Big),
 \end{aligned}
\end{align}
where the unscaled coefficients are $\widetilde{A}_{ijk}\coloneqq A_{ijk}/(2i+1)$ and $\widetilde{B}_{ijk} \coloneqq B_{ijk}/(2i+1)$. The following lemma shows that $\mathcal{Q}$ can be written in conservative form, namely as the (spatial) partial derivative of a function.
\begin{lemma} \label{lem:Q:conservative}
The term $\mathcal{Q}$ can be written in conservative form as follows
\begin{align} \label{eq:SWME_conservative_coef}
\mathcal{Q}(x,t) = \partial_x\widehat{Q}(x,t),
\end{align}
where
\begin{align} \label{eq:SWME_defQhat}
\widehat{\mathcal{Q}}(x,t)\coloneqq  \sum\limits_{i,j,k=1}^N  \left( \widetilde{A}_{i j k} + \widetilde{B}_{i j k}\right) h(x,t) \alpha_i(x,t) \alpha_j(x,t) \alpha_k(x,t) .
\end{align}
\end{lemma}

\begin{proof}
Computing the partial derivative of $\widehat{\mathcal{Q}}$ with respect to $x$ in \eqref{eq:SWME_defQhat} and applying product rule, we have
\begin{align*}
  \partial_x\widehat{\mathcal{Q}}
  & =
  \sum\limits_{i,j,k=1}^N \bigg\{\!\left( \widetilde{A}_{i j k} + \widetilde{B}_{i j k}\right)\alpha_i\alpha_k\partial_x (h\alpha_j)   + \left( \widetilde{A}_{i j k} + \widetilde{B}_{i j k} + \widetilde{A}_{kji} + \widetilde{B}_{kji}\right)h\alpha_i\alpha_j\partial_x \alpha_k\bigg\}.
 \end{align*}
We observe that the right-hand side of the above equation is equal to $\mathcal{Q}$ \eqref{eq:coefs} provided the following identity holds
\begin{align}
     \widetilde{B}_{ijk} +  \widetilde{A}_{kji} + \widetilde{B}_{kji} & = 0.
     \label{eq:condition}
\end{align}
As shown in \ref{sec:appendix:A}, the condition \eqref{eq:condition} is true for all $i,j,k=1,\dots,N$, which concludes the proof.
\end{proof}

Then, we define the vector of conservative variables as $\bB{u} = (h,hu_m,h\alpha_1,h\alpha_2,\dots,h\alpha_N)^{\rm T}\in \mathbb{R}^{N+2}$ for the SWME \eqref{sys:SWME} (or SWLME \eqref{sys:SWLME}), which in short notation can be written as $\bB{u} = \big(h,hu_m,h\bB{\alpha}^{\rm T}\big)^{\rm T}$. We are now in position to establish an equation for the total energy in the SWME model, and further determine the associated flux related to the energy balance.
\begin{theorem}\label{thm:tot:energy}
Let $h$, $u_m$ and $\bB{\alpha}$ be a smooth solution of the SWME~\eqref{sys:SWME}. Then, the total energy $\mathbb{E}$ is a conserved quantity with total energy equation
\begin{align}\label{eq:total:energy:conservation}
 \partial_t \mathbb{E}(\bB{u}) + \partial_x \mathbb{F}(\bB{u}) = 0,
\end{align}
where the total energy $\mathbb{E}$ and total flux $\mathbb{F}$ are defined as
\begin{align}
& \mathbb{E}(\bB{u}) \coloneqq \frac{h u_m^2}{2}+ \frac{ h}{2}\sum\limits_{i=1}^N\frac{\alpha_i^2}{2i+1}+ g \frac{h^2}{2}+g h b, \label{eq:Total:Energy}\\
& \begin{aligned}
\mathbb{F}(\bB{u}) & \coloneqq \frac{h u_m^3}{2} + h u_m \frac{3}{2} \sum\limits_{i=1}^N \frac{\alpha_i^2}{2i+1} + g h u_m (h+b) + \sum\limits_{i,j,k=1}^N  \frac{1}{2i+1} \left( A_{i j k} + B_{i j k}\right) h \alpha_i \alpha_j \alpha_k.
\end{aligned}\label{eq:Total:Flux}
\end{align}
\end{theorem}
\begin{proof}
We first obtain the total kinetic energy equation \eqref{eq:SWME_KtotalFinal} by adding the averaged kinetic equation \eqref{eq:SWME_K} from Lemma~\ref{lem:partial:energy} to Equation~\eqref{eq:SWME_Psi} in Lemma~\ref{lem:Psi:equation}, and apply the result from Lemma~\ref{lem:Q:conservative} to have
\begin{align}
    \partial_t\left(\frac{h u_m^2}{2}+ \dfrac{\Psi}{2}\right) + \partial_x \left(\frac{h u_m^3}{2}\right) +g h u_m \partial_x(h+b) + \partial_x \left(  \frac{3}{2} u_m\Psi + \widehat{\mathcal{Q}}\right) & = 0, \label{eq:SWME_KtotalFinal}
\end{align}
where $\Psi$ is defined in \eqref{eq:SWLME_defPsi}, and $\widehat{\mathcal{Q}}$ is defined in \eqref{eq:SWME_defQhat}. The potential equation, on the other hand, is given by Equation~\eqref{eq:SWME_P}, in Lemma~\ref{lem:partial:energy}. Then, combining \eqref{eq:SWME_KtotalFinal} and \eqref{eq:SWME_P}, we obtain the desired total energy equation \eqref{eq:total:energy:conservation}.
\end{proof}

Note that the total energy function $\mathbb{E}$ in Theorem~\ref{thm:tot:energy} acts as a convex mathematical entropy function of system~\eqref{sys:SWME} with corresponding entropy flux $\mathbb{F}$, where the convexity can be readily verified from positive definiteness of the Hessian of $\mathbb{E}$, which is shown in Lemma~\ref{lemma:spd:hessian} in \ref{sec:appendix:C}. The next corollary introduces the entropy variables $\bB{w}$ together with the entropy Hessian $\mbf{H}^{-1}$ related to \eqref{sys:SWME}.

\begin{corollary} \label{cor:entropy:variables}
 The entropy variables, denoted by $w_1,w_2,\dots,w_{N+2}$, associated to the SWME system \eqref{sys:SWME} are:
 \begin{align}
    w_1 & = -\frac{u_m^2}{2}- \frac{1}{2}\sum\limits_{i=1}^N\frac{\alpha_i^2}{2i+1}+g(h+b), \qquad
    w_2 = u_m,\qquad
    w_{i+2} = \frac{\alpha_i}{2i+1}, \label{eq:SWME:Entropy:Variables}
\end{align}
for all $i=1,\dots,N$. Furthermore, the Hessian of $\mathbb{E}$ with respect to the conservative variables $\bB{u}$ is given by
\begin{equation}\label{eq:entropy:hessian}
    \mathbf{H^{-1}} = \frac{1}{h}\begin{pmatrix}
        gh + u_m^2 + \sum\limits_{i=1}^N \frac{\alpha_i^2}{2i + 1} & -u_m & - \frac{\alpha_1}{3} & \dots & -\frac{\alpha_N}{2N + 1} \\
        -u_m & 1 & 0 & \dots & 0 \\
        -\frac{\alpha_1}{3} & 0 & \frac{1}{3} & \ddots & \vdots \\
        \vdots & \vdots & \ddots & \ddots & 0 \\
        -\frac{\alpha_N}{2N+1} & 0 & \dots & 0 & \frac{1}{2N + 1}
    \end{pmatrix}.
\end{equation}
\end{corollary}

\begin{proof}
 The entropy variables are defined as the partial derivative of the entropy function $\mathbb{E}$ with respect to the conservative variables $(h,hu_m,h\bB{\alpha})$. Therefore computing $w_1=\partial_h \mathbb{E}$, $w_2 = \partial_{hu_m}\mathbb{E}$ and $w_{i+2} = \partial_{h\alpha_i} \mathbb{E}$ for all $i=1,\dots,N$, yields the entropy variables. The Hessian of $\mathbb{E}$ is then obtained as the Jacobian of the entropy variables, with respect to the conservative variables.
\end{proof}

\begin{remark}
Notice that the energy function is independent of the coefficients $A_{ijk}$ and $B_{ijk}$. Thus, the entropy function and entropy variables corresponding to the linearized model SWLME \eqref{sys:SWLME} are the same function $\mathbb{E}$ in \eqref{eq:Total:Energy} and vector $\bB{w}$ from Corollary~\ref{cor:entropy:variables}, respectively. The total flux, which also corresponds to the entropy flux for the SWLME and was first derived in \cite{Caballero2025} and later more systematically in \cite{KoellermeierEnergy2025arxiv} is, on the other hand, given by 
\begin{align*}
\mathbb{F}_{\tt LM}(\bB{u}) & \coloneqq \frac{h u_m^3}{2} + h u_m \frac{3}{2} \sum\limits_{i=1}^N \frac{\alpha_i^2}{2i+1} + g h u_m (h+b).
\end{align*}
\end{remark}

\smallskip
In the next section, we discuss the case of including the friction effect, which turns out to be a source term in the model equations, and establish that two widely used friction terms described in the literature are indeed entropy dissipative with respect to the developed entropy variables. 

\subsection{Entropy dissipation for the friction term}\label{sec:Entropy:dissipation}

When friction terms are taken into account, the SWME \eqref{sys:SWME} (or SWLME \eqref{sys:SWLME}) become non-homogeneous and a source term is included. For classical Newtonian slip friction, this source term is defined as \cite{Kowalski2019}
\begin{equation}\label{eq:friction:Nslip}
	\mbf{S}_{\rm Ns}(\bB{u}) \coloneqq
	-\frac{\nu}{\lambda}u_b
	\begin{pmatrix}
		0 \\
		1\\
		\bB{r}\\
	\end{pmatrix}
	- \dfrac{\nu}{h}
	\begin{pmatrix}
		0 \\
		0\\
		\bB{C}\bB{\alpha}\\
	\end{pmatrix},
\end{equation}
where $u_b := u_m(x,t) + \alpha_1(x,t) + \alpha_2(x,t) + \cdots + \alpha_N(x,t)$ is the bottom velocity, the vector $\bB{r}\in\mathbb{R}^N$ is defined componentwise, where each component is given by $r_i = 2i+1$ for all $i=1,\dots,N$, and the matrix
${\bB{C}}=({C}_{ij})\in \mathbb{R}^{N\times N}$ is defined for each component as
\begin{align}
 \label{eq:CMatrix}
 {C}_{ij} \coloneqq (2i+1)\int_0^1\phi_i'(\zeta)\phi_j'(\zeta)\,{\rm d}\zeta\qquad\text{for }i,j=1,\dots,N.
\end{align}
The constant $\lambda>0$ is the so-called slip length, and $\nu>0$ is a viscosity related parameter. The first term in \eqref{eq:friction:Nslip} corresponds to the bottom friction given by the slip condition, while the second term containing the $\bB{C}$ matrix is related to the bulk friction.

In the next lemma, we show that the friction term $\mbf{S}_{\rm Ns}$ is in fact entropy dissipative with respect to the entropy variables $\bB{w}$ defined in Corollary~\ref{cor:entropy:variables}.

\begin{lemma} \label{lem:PNs:dissipative}
 The friction term $\mbf{S}_{\rm Ns}$ defined by \eqref{eq:friction:Nslip} is entropy dissipative with respect to $\bB{w}$, namely
 \begin{align}
  \bB{w}^{\rm T}\mbf{S}_{\rm Ns}(\bB{u})\leq 0.
 \end{align}
\end{lemma}
\begin{proof}
Multiplying the entropy vector $\bB{w}$ to the friction term $\mbf{S}_{\rm Ns}$ we have
\begin{equation}
	\begin{aligned}
	\bB{w}^{\rm T}\mbf{S}_{\rm Ns} &= -\frac{\nu}{\lambda}u_b^2-\dfrac{\nu}{h}\sum_{i,j=1}^N\left(\int_0^1 \phi_i'(\zeta)\phi_j'(\zeta){\rm d}\zeta\right)\alpha_i \alpha_j\\
	&=
	-\frac{\nu}{\lambda}u_b^2
	- \frac{\nu}{h} \int_0^1 \bigg(\sum_{i=1}^N\phi_i'(\zeta) \alpha_i\bigg) \bigg(\sum_{j=1}^N \phi_j'(\zeta) \alpha_j \bigg) {\rm d}\zeta
	\\&=
	-\frac{\nu}{\lambda} u_b^2
	- \frac{\nu}{h} \int_0^1 \bigg(\sum_{i=1}^N\phi_i'(\zeta) \alpha_i\bigg)^2 {\rm d}\zeta \leq 0,
	\end{aligned}
\end{equation}
which completes the proof.
\end{proof}
\noindent An alternative friction term can be obtained by using a Manning law for the bottom friction \cite{Manning1891} as in \cite{Garres-Diaz2021a}, which is given by
\begin{equation}\label{eq:friction:NManning}
	\mbf{S}_{\rm NM}(\bB{u}) \coloneqq
	-\frac{\rho g n^2}{h^{1/3}}\big| u_b \big| u_b
	\begin{pmatrix}
		0 \\
		1\\
		\bB{r}\\
	\end{pmatrix}
	- \dfrac{\nu}{h}
	\begin{pmatrix}
		0 \\
		0\\
		\bB{C}\bB{\alpha}\\
	\end{pmatrix},
\end{equation}
where $n>0$ is called the Manning coefficient. In the next lemma, we show that $\mbf{S}_{\rm NM}$ is also entropy dissipative.

\begin{lemma} \label{lem:PNM:dissipative}
 The friction term $\mbf{S}_{\rm NM}$ defined by \eqref{eq:friction:NManning} is entropy dissipative with respect to $\bB{w}$.
\end{lemma}
\begin{proof}
Multiplying the entropy vector $\bB{w}$ to the friction term $\mbf{S}_{\rm NM}$, we perform the same calculations done in the proof of Lemma~\ref{lem:PNs:dissipative} to obtain
\begin{equation}
	\begin{aligned}
	\bB{w}^{\rm T}\mbf{S}_{\rm NM} &= 
	-\frac{\rho g n^2}{h^{1/3}} \left|u_b\right| u_b^2
	- \frac{\nu}{h} \int_0^1 \bigg(\sum_{i=1}^N\phi_i'(\zeta) \alpha_i\bigg)^2 {\rm d}\zeta \leq 0,
	\end{aligned}
\end{equation}
which completes the proof.
\end{proof}

\section{Numerical method}\label{sec:numerical:scheme}

Before delving further into the numerical approximation of the SWME, we augment the system with the additional trivial evolution equation $b_t = 0$. This reformulation incorporates the bottom topography source term into the nonconservative product, which simplifies the construction of well-balanced schemes \cite{munoz2011convergence}. 
To this end, we augment both vectors of conservative variables $\bB{u} = (h,hu_m,h\alpha_1,h\alpha_2,\dots,h\alpha_N,b)^{\rm T}\in \mathbb{R}^{N+3}$ as well as entropy variables $\bB{w} = (w_1,w_2,\dots,w_{N+2},gh)^{\rm T} \in \mathbb{R}^{N+3}$ to account for the evolution equation of the bottom.
The resulting augmented and non-homogeneous system~\eqref{sys:SWME} is then written in vector notation as follows
\begin{align} \label{sys:SWME:vectorial}
 \partial_t \bB{u} + \partial_x \mbf{f}(\bB{u}) + \mbf{B}(\bB{u}) \partial_x \bB{u} = \mbf{S}(\bB{u}),
\end{align}
where $\mbf{S} = \mbf{S}(\bB{u})$ denotes a source term, which typically includes the friction related contributions, and
\begin{align*}
 \mbf{f}(\bB{u}) = \begin{pmatrix}
                    hu_m\\
                    hu_m^2 + \Psi\\
                    2h u_m\bB{\alpha} + \bB{\mathfrak{A}}\\
                    0
                   \end{pmatrix},\qquad
\mbf{B}(\bB{u}) =
\underbrace{\begin{pmatrix}
    0 & 0 & \bB{0}^{\rm T} & 0\\
       gh & 0 & \bB{0}^{\rm T} & gh\\
        \bB{0} & \bB{0} & -u_m\mbf{I} & \bB{0} \\
 0 & 0 & \bB{0}^{\rm T} & 0
  \end{pmatrix}}_{\mathlarger{\mbf{B}^{\tt LM}(\bB{u})}}
  +
\underbrace{\begin{pmatrix}
 0 & 0 & \bB{0}^{\rm T} & 0\\
 0 & 0 & \bB{0}^{\rm T} & 0\\
 \bB{0} & \bB{0} & \widehat{\mbf{B}}(\bB{\alpha}) & \bB{0} \\
 0 & 0 & \bB{0}^{\rm T} & 0
\end{pmatrix}}_{\mathlarger{\mbf{\bar{B}}^{\tt LM}(\bB{u})}},
\end{align*}
with $\bB{0}$ the (column) vector of zeros in $\mathbb{R}^N$, $\mbf{I}$ the identity matrix in $\mathbb{R}^{N\times N}$, and $\widehat{\mbf{B}}(\bB{\alpha}) \in \mathbb{R}^{N\times N}$ is the matrix whose elements are defined by
\begin{align*}
  \widehat{B}_{ij}(\bB{\alpha}) = B_{ij1}\alpha_1+B_{ij2}\alpha_2+\cdots + B_{ijN}\alpha_N.
\end{align*}
The first summand in the definition of $\mbf{B}$ corresponds to the nonconservative coefficient matrix related to the linearized moment model SWLME, and is denoted by $\mbf{B}^{\tt LM}$ while the second summand that involves the coefficients $B_{ijk}$ is denoted by $\mbf{\bar{B}}^{\tt LM}$. 

To evaluate the moment coefficients $A_{ijk}$ and $B_{ijk}$ in \eqref{eq:Aijk_Bijk}, and friction coefficients ${C}_{ij}$ in \eqref{eq:CMatrix}, we construct the shifted Legendre polynomials \eqref{eq:shifted_legendre_polynomial} from a recurrence relation and apply Legendre-Gauss quadrature for accurate integration, see \ref{sec:computeTensors} for details.
Furthermore, we introduce the standard jump and average operators, denoted by $\jump{\cdot}$ and $\avg{\cdot}$, respectively, which for a given function $\varphi=\varphi(x)$ are defined by
\begin{align*}
 \jump{\varphi} := \jump{\varphi} (x) =  \varphi^+ - \varphi^-,\qquad  \avg{\varphi}:=\avg{\varphi}(x) = \tfrac{1}{2}(\varphi^+ + \varphi^-),
\end{align*}
where $\varphi^{+} = \varphi(x^+)$ and $\varphi^{-} = \varphi(x^-)$ for all $x\in \Omega$.

In the following sections, we first introduce an entropy conservative numerical scheme that recovers an integral version of the entropy conservation law \eqref{eq:total:energy:conservation} in the semi-discrete case.
This is achieved by formulating a numerical scheme that satisfies a summation-by-parts (SBP) property and the derivation of entropy conservative numerical fluxes to recover a semi-discrete version of the relation
\begin{equation*}
	\int_{\Omega}
	\bB{w}^{\rm T}\big(\partial_t\bB{u} + \partial_x \mbf{f}(\bB{u}) + \mbf{B}(\bB{u}) \partial_x \bB{u}\big)\, {\rm d}x 
	= \int_{\Omega}  \big(\partial_t \mathbb{E}(\bB{u}) + \partial_x \mathbb{F}(\bB{u})\big)\,{\rm d}x = 0.
\end{equation*}
This entropy conservative scheme acts as a baseline that is further extended to one that is entropy stable, by adding suitable numerical dissipation at interfaces, in the sense that the numerical method satisfies the entropy inequality
\begin{equation}\label{eq:entropy_inequality}
	\int_{\Omega} \big(\partial_t \mathbb{E}(\bB{u}) + \partial_x \mathbb{F}(\bB{u}) \big)\,\text{d}x \leq 0,
\end{equation}
which accounts for entropy dissipation at discontinuities. 


\subsection{Entropy conservative discontinuous Galerkin scheme} \label{sec:DG:description} 

For the spatial approximation, we formulate a collocated nodal discontinuous Galerkin spectral element method (DGSEM) on Legendre–Gauss–Lobatto (LGL) nodes \cite{Kopriva2009,Hesthaven2007}. 
To account for the nonconservative terms of the system, we formulate a path-conservative method that is written in terms of fluctuations as  established in the work of Par{\'e}s~\cite{Pares2006} for FV methods.
Entropy stable methods in fluctuation form have first been developed in \cite{castro2013entropy} in the context of FV schemes, by incorporating the ideas of Tadmor~\cite{Tadmor1987}. 
Further extensions to ES-DGSEM that combine a volume integral in flux differencing form with EC volume fluctuations were proposed in \cite{ersing2026new, Renac2019, coquel2021entropy, waruszewski2022entropy}, where in the following work we consider the formulation from \cite{ersing2026new}.

To apply this method, we consider a uniform discretization of the domain $\Omega$ into $k=1,...,K$ non-overlapping elements of size $\Delta x_k$.
Each element is mapped onto the reference domain $E = \{\xi \in \mathbb{R} \; : \; |\xi| \leq 1 \}$, where the solution is approximated in the space of polynomials up to degree $P$
\begin{equation*}\label{eq:polynomial_approximation}
	\bB{u} \approx \mbf{U}^k(\xi) = \sum_{j=0}^P \mbf{U}^k_j l_j(\xi),
\end{equation*}
spanned by nodal Lagrange basis functions $\{l_j(\xi)\}_{j=0}^P$ with interpolation nodes located at LGL points and test functions are taken from the basis.  
We further introduce the discrete derivative operator $\mathcal{D} \in \mathbb{R}^{(P+1) \times (P+1)}$, that is obtained from the derivatives of the Lagrange basis functions, and approximate integrals with LGL quadrature formulas, with corresponding LGL quadrature weights $\{\omega_j\}_{j=0}^P$. This choice of discrete operators collocates the interpolation and quadrature nodes and ensures that the method is endowed with a diagonal norm summation-by-parts property \cite{gassner2013skew}, which is key to ensure entropy stability.
The semi-discrete nonconservative DGSEM in flux-differencing form for the degree of freedom $i = 0,1,\dots,P$ is then given by \cite{ersing2026new}
\begin{equation}\label{eq:dgsem}
\omega_i \frac{\Delta x_k}{2} \partial_t \mbf{U}_i^k + \omega_i \sum\limits_{m=0}^P 2\mathcal{D}_{im} \mbf{D}^{-}(\mbf{U}_i^k, \mbf{U}_m^k) + \delta_{i0}\mbf{D}^{+}(\mbf{U}_P^{k-1}, \mbf{U}_0^k) + \delta_{iP} \mbf{D}^{-}(\mbf{U}_P^k, \mbf{U}_0^{k+1}) = \mbf{0},
\end{equation}
where $\delta_{i0}$ and $\delta_{iP}$ denote the Kronecker delta for the indices $0$ and $P$, respectively, and $\mbf{D}^{\pm}$ are fluctuations in the context of path-conservative schemes as introduced in \cite{Pares2006}, satisfying the consistency condition $\mbf{D}^{\pm}(\bB{u},\bB{u}) = \bB{0}$ and a path-conservation condition
\begin{equation}\label{eq:path_conservation_cond}
	\mbf{D}^{-}(\bB{u}^-, \bB{u}^+) + \mbf{D}^+(\bB{u}^-, \bB{u}^+) = \jump{\mbf{f}} + \int_0^1 \mbf{B}(\Phi(s)) \partial_s \Phi(s) \, \text{d}s,
\end{equation}
where $\Phi(s) \coloneqq \Phi(s; \bB{u}^-, \bB{u}^+)$ is a parametrization of the path connecting the states $\bB{u}^-$ and $\bB{u}^+$. In the present work, we consider fluctuations that are derived by taking the linear path $\Phi(s) = \bB{u}^- + s \jump{\bB{u}}$ and approximating the path integral in \eqref{eq:path_conservation_cond} with the trapezoidal rule to obtain
\begin{equation}\label{eq:fluctuation_definition}
	\mbf{D}^{\pm}(\bB{u}^-, \bB{u}^+) = \frac{1}{2} \mbf{B}(\bB{u}^{\pm})\jump{\bB{u}} \pm \left( \mbf{f}(\bB{u}^{\pm}) - \mbf{f}^*(\bB{u}^-, \bB{u}^+)\right),
\end{equation}
where $\mbf{f}^*$ denotes a numerical flux function that is consistent with the physical flux $\mbf{f}^*(\bB{u}, \bB{u}) = \mbf{f}(\bB{u})$.
As shown in \cite[Lemma 2]{ersing2026new} we can construct an entropy conservative (EC) scheme from \eqref{eq:dgsem}, i.e., a method that recovers a semi-discrete version of the entropy conservation law \eqref{eq:total:energy:conservation}, if the fluctuations \eqref{eq:fluctuation_definition} are constructed from an EC numerical flux $\mbf{f}^{\tt EC}:=\mbf{f}^{\tt EC}(\bB{u}^-,\bB{u}^+)$ that satisfies the entropy conservation condition
\begin{equation}
	\avg{\bB{w}^{\rm T}\mbf{D}} = \jump{\mathbb{F}} \quad \Leftrightarrow \quad
	\jump{\bB{w}}^{\rm T} \mbf{f}^{\tt EC} - \avg{\bB{w}^{\rm T} \mbf{B}} \jump{\bB{u}}  = \jump{\bB{w}^{\rm T}\mbf{f} - \mathbb{F}},
	\label{eq:entropy_conservation_condition}
\end{equation}
where the term on the right-hand side of \eqref{eq:entropy_conservation_condition} is the entropy potential.
Accordingly, we define the EC fluctuations
\begin{equation}\label{eq:ec_fluctuation_definition}
	\mbf{D}_{\tt EC}^{\pm}(\bB{u}^-, \bB{u}^+) = \frac{1}{2} \mbf{B}(\bB{u}^\pm) \jump{\bB{u}} \pm \left(\mbf{f}(\bB{u}^{\pm}) - \mbf{f}^{\tt EC}(\bB{u}^-, \bB{u}^+)\right).
\end{equation}

In the following section, we demonstrate how to construct such an EC numerical flux $\mbf{f}^{\tt EC}$ for the SWME and SWLME. In this procedure we will exploit the hierarchical character of the SWME to decompose the system into components that satisfy distinct entropy conservation conditions. 

\begin{remark}
    For all non-conservative systems the solution depends on the chosen path \cite{Pares2006}. In the case of the SWME, however, similar to works on moment models for rarefied gases where the dependence on the choice of the path was investigated in great detail, see e.g. \cite{koellermeier_numerical_2017}, the linear path gives accurate results. There is little to no practical dependence on the path. This can be attributed to the occurrence of the non-conservative products mainly in the higher order moment equations, while the mass and momentum equations do not contain non-conservative terms. 
\end{remark}

\subsection{Entropy conservative flux for the SWME} \label{sec:EC-Moment-Fluxes}

In order to construct an EC flux that satisfies \eqref{eq:entropy_conservation_condition}, we first observe that both the SWE and SWLME are part of the algebraic definition of the SWME. The conservative fluxes of the SWLME can be constructed hierarchically by defining the respective fluxes as follows
\begin{align}
\mbf{f}^{\tt SWE}(\bB{u}) \coloneqq \begin{pmatrix}
                    hu_m\\
                    hu_m^2\\
                   \bB{0}\\
                    0
                   \end{pmatrix},
                   \qquad
\mbf{\bar{f}}^{\tt SWE}(\bB{u}) \coloneqq \begin{pmatrix}
                    0\\
                    \Psi\\
                   2hu_m\bB{\alpha}\\
                    0
                   \end{pmatrix},
\qquad
\mbf{f}^{\tt LM}(\bB{u}) \coloneqq \mbf{f}^{\tt SWE}(\bB{u}) +  
\mbf{\bar{f}}^{\tt SWE}(\bB{u}).\qquad
\end{align}
In addition, it is straightforward to write the flux corresponding to the SWME in \eqref{sys:SWME:vectorial} in terms of $\mbf{f}^{\tt LM}$ with the equation $ \mbf{f}(\bB{u}) = \mbf{f}^{\tt LM} (\bB{u}) + \mbf{\bar{f}}^{\tt LM}(\bB{u})$,
where the additional flux contribution is defined by $\mbf{\bar{f}}^{\tt LM}(\bB{u}) \coloneqq (0,0,\bB{\mathfrak{A}}^{\rm T},0)^{\rm T}$.
In a similar fashion, the nonconservative term exhibits the same hierarchical structure
\begin{align} \label{eq:BSWE:BLM}
 \mbf{B}^{\tt SWE}(\bB{u}) \coloneqq
\begin{pmatrix}
    0 & 0 & \bB{0}^{\rm T} & 0\\
       gh & 0 & \bB{0}^{\rm T} & gh\\
        \bB{0} & \bB{0} & \bB{\Theta} & \bB{0} \\
 0 & 0 & \bB{0}^{\rm T} & 0
  \end{pmatrix},\qquad
  \mbf{\bar{B}}^{\tt SWE}(\bB{u}) \coloneqq 
  \begin{pmatrix}
    0 & 0 & \bB{0}^{\rm T} & 0\\
       0 & 0 & \bB{0}^{\rm T} & 0\\
        \bB{0} & \bB{0} & -u_m\mbf{I} & \bB{0} \\
 0 & 0 & \bB{0}^{\rm T} & 0
  \end{pmatrix}, \qquad
  \mbf{B}^{\tt LM}(\bB{u}) \coloneqq  \mbf{B}^{\tt SWE}(\bB{u}) + \mbf{\bar{B}}^{\tt SWE}(\bB{u}), \quad
\end{align}
where $\bB{\Theta}$ is the zero matrix in $\mathbb{R}^{N\times N}$. 
Then, the nonconservative term related to the SWME in \eqref{sys:SWME:vectorial} satisfies that $\mbf{B}(\bB{u}) = \mbf{B}^{\tt LM}(\bB{u}) + \mbf{\bar{B}}^{\tt LM}(\bB{u})$. The hierarchical property of the moment models allows us to compute the potentials as follows
\begin{align}\label{eq:potential:hierarchy}
\begin{aligned}
 \jump{\bB{w}^{\rm T}\mbf{f} - \mathbb{F}} & = \jump{\bB{w}^{\rm T}(\mbf{f}^{\tt LM} + \mbf{\bar{f}}^{\tt LM}) - (\mathbb{F}^{\tt LM} + \bar{\mathbb{F}}^{\tt LM})}\\
 & = \jump{\bB{w}^{\rm T}\mbf{f}^{\tt LM} - \mathbb{F}^{\tt LM}} +
     \jumpbig{\bB{w}^{\rm T}\mbf{\bar{f}}^{\tt LM} - \overline{\mathbb{F}}^{\tt LM}\,},
\end{aligned}
\end{align}
where $\mathbb{F}^{\tt LM}$ is the entropy flux of the SWLME defined as the entropy flux $\mathbb{F}$ in \eqref{eq:Total:Flux} without the terms related to $A_{ijk}$ and $B_{ijk}$, and the remaining flux is $\overline{\mathbb{F}}^{\tt LM} \coloneqq \mathbb{F} - \mathbb{F}^{\tt LM}$. Hence, we compute the entropy potentials
\begin{align}
&		\begin{aligned} \label{eq:potential:LM}
			\mathcal{P}^{\tt LM} = \bB{w}^{\rm T}\mbf{f}^{\tt LM} - \mathbb{F}^{\tt LM}
			&= \left(g(h+b) - \tfrac{1}{2}u_m^2 - \tfrac{1}{2} \Psi/h\right) hu_m+
			u_m\left(hu_m^2 + \Psi\right)
			+ 2u_m \Psi - \mathbb{F}^{\tt LM}
			\\&=
			g(h+b)hu_m + \frac{1}{2}hu_m^3 + \frac{5}{2}u_m\Psi - \mathbb{F}^{\tt LM} = u_m\Psi,
		\end{aligned}\\
&		\begin{aligned}\label{eq:potential:LM:bar}
			\overline{\mathcal{P}}^{\tt LM} = \bB{w}^{\rm T}\mbf{\bar{f}}^{\tt LM} - \overline{\mathbb{F}}^{\tt LM}
			& =  \sum_{i=1}^N\dfrac{\alpha_i \mathfrak{A}_i}{2i+1} - \overline{\mathbb{F}}^{\tt LM}
			 =  -\sum_{i, j, k=1}^N \dfrac{B_{i j k}}{2i+1} h\alpha_i \alpha_j \alpha_k.
		\end{aligned}
\end{align}
Finally, the entropy potential related to the SWME given in \eqref{eq:entropy_conservation_condition} is given by
\begin{align}
 \jump{\bB{w}^{\rm T}\mbf{f} - \mathbb{F}} = \jump{\mathcal{P}^{\tt LM}} + \jump{\overline{\mathcal{P}}^{\tt LM}} =  \jump{u_m \Psi} - \sum_{i, j, k=1}^N \dfrac{B_{i j k}}{2i+1} \jump{h\alpha_i \alpha_j \alpha_k}.
\end{align}
With the entropy potentials determined, we are equipped to derive the EC flux for the SWME system \eqref{sys:SWME}, i.e., a numerical flux that satisfies \eqref{eq:entropy_conservation_condition}. To this end, we will make use of the hierarchical property of the SWME system.
From \cite{Ersing2025}, we know that a suitable entropy conservative flux for the SWE is the following
\begin{align}
 \mbf{f}^{\tt EC, SWE}(\bB{u}^-,\bB{u}^+) =
		\begin{pmatrix}
			\avg{hu_m} \\
			\avg{hu_m}\avg{u_m} \\
			\bB{0}\\
			0
		\end{pmatrix},
\end{align}
which together with the nonconservative matrix $\mbf{B}^{\tt SWE}$ defined in \eqref{eq:BSWE:BLM}, and the flux $\mbf{f}^{\tt EC, SWE}$  satisfies the following equation:
	\begin{equation}
		\jump{\bB{w}}^{\rm T} \mbf{f}^{\tt EC,SWE} - \avg{\bB{w}^{\rm T} \mbf{B}^{\tt SWE}} \jump{\bB{u}} = -\frac{1}{2}\sum\limits_{j=1}^N \frac{1}{2j + 1}\jump{\alpha_j^2} \avg{hu_m}.
		\label{eq:entropy_conservation_condition_swe}
	\end{equation}
Next, we discretize the remaining fluxes $\mbf{\bar{f}}^{\tt SWE}$ and $\mbf{\bar{f}}^{\tt LM}$ by
\begin{equation}
 \mbf{\bar{f}}^{\tt EC, SWE}(\bB{u}^-,\bB{u}^+) \coloneqq
		\begin{pmatrix}
			0 \\
			\sum\limits_{j=1}^N \frac{\avg{h\alpha_j}\avg{\alpha_j}}{2j + 1} \\
			\avg{hu_m}\avg{\bB{\alpha}} + \avg{h\bB{\alpha}}\avg{u_m}\\
			0
		\end{pmatrix},\qquad
		\mbf{\bar{f}}^{\tt EC,LM}(\bB{u}^-,\bB{u}^+)\coloneqq
		\begin{pmatrix}
			0 \\
			0 \\
			\sum\limits_{j,k=1}^N \avg{h\alpha_j}\avg{\alpha_k}\mathbf{A}_{jk} \\
			0
		\end{pmatrix},
		\label{eq:swme_ec_fluxes}
	\end{equation}
where for all $i,j = 1,\dots,N$, the vector $\mbf{A}_{jk} = (A_{1jk}, A_{2jk},\dots,A_{Njk})^{\rm T}$ is in $\mathbb{R}^N$. We are now able to define our entropy conservative numerical flux for the SWME system \eqref{sys:SWME:vectorial}
\begin{align} \label{eq:f:EC}
 \mbf{f}^{\tt EC}(\bB{u}^-,\bB{u}^+) \coloneqq \mbf{f}^{\tt EC, SWE}(\bB{u}^-,\bB{u}^+) + \mbf{\bar{f}}^{\tt EC, SWE}(\bB{u}^-,\bB{u}^+) + \mbf{\bar{f}}^{\tt EC,LM}(\bB{u}^-,\bB{u}^+).
\end{align}
The following theorem shows that the numerical flux $\mbf{f}^{\tt EC}$ defined by Equation~\eqref{eq:f:EC} is entropy conservative, i.e., it satisfies the entropy condition \eqref{eq:entropy_conservation_condition}.
\begin{theorem}\label{thm:ec_flux}
 The numerical flux $\mbf{f}^{\tt EC}$ defined in \eqref{eq:f:EC}, which approximates the flux $\mbf{f}$ in system~\eqref{sys:SWME:vectorial} fulfills the entropy condition \eqref{eq:entropy_conservation_condition} with the entropy flux $\mathbb{F}$ defined in \eqref{eq:Total:Flux}.
\end{theorem}
\begin{proof}
In the following, we make use of the hierarchical property of the SWME which is described in the previous section. Thanks to the linearity of the average operator $\avg{\cdot}$ and the definition of $\mbf{f}^{\tt EC}$ and $\mbf{B}$ as sums of fluxes and nonconservative matrices, respectively, the left-hand side of Equation~\eqref{eq:entropy_conservation_condition} can be written as
\begin{align}
\begin{aligned} \label{eq:leftside:EC:linearity}
 \jump{\bB{w}}^{\rm T} \mbf{f}^{\tt EC} - \avg{\bB{w}^{\rm T} \mbf{B}} \jump{\bB{u}} & =
 \jump{\bB{w}}^{\rm T} \big(\mbf{f}^{\tt EC, SWE} + \mbf{\bar{f}}^{\tt EC, SWE} + \mbf{\bar{f}}^{\tt EC, LM}\big)
 - \avg{\bB{w}^{\rm T} \big(\mbf{B}^{\tt SWE} + \mbf{\bar{B}}^{\tt SWE} + \mbf{\bar{B}}^{\tt LM}\big)} \jump{\bB{u}}\\
 & =
 \jump{\bB{w}}^{\rm T} \mbf{f}^{\tt EC, LM}
 - \avg{\bB{w}^{\rm T} \mbf{B}^{\tt LM}} \jump{\bB{u}} +
 \jump{\bB{w}}^{\rm T} \mbf{\bar{f}}^{\tt EC, LM}
 - \avg{\bB{w}^{\rm T} \mbf{\bar{B}}^{\tt LM}} \jump{\bB{u}}.
\end{aligned}
\end{align}

Then, we compute the terms related to the linearized moment model as follows:

\begin{align}
\begin{aligned}\label{eq:leftside:EC:P}
 \jump{\bB{w}}^{\rm T} \mbf{f}^{\tt EC, LM}
 - \avg{\bB{w}^{\rm T} \mbf{B}^{\tt LM}} \jump{\bB{u}}
			& = \sum_{j=1}^N \frac{1}{2j+1}\Big(
		 \jump{u_m}\avg{h\alpha_j}\avg{\alpha_j} + \jump{\alpha_j}\avg{h\alpha_j}\avg{u_m}
			+
			\avg{\alpha_j u_m}\jump{h\alpha_j}\Big)\\
			& = \sum\limits_{j=1}^N \frac{1}{2j+1}
			\jump{hu_m\alpha_j^2} =\jump{\mathcal{P}^{\tt LM}},
 \end{aligned}
\end{align}
where $\mathcal{P}^{\tt LM}$ is defined in \eqref{eq:potential:LM:bar}, and we have used the identity
\begin{align*}
\frac{1}{2}\jump{\alpha_j^2} \avg{hu_m} =  \jump{\alpha_j}\avg{hu_m}\avg{\alpha_j}.
\end{align*}
On the other hand, for the full moment model related terms in the right-hand side of \eqref{eq:leftside:EC:linearity}, we use the identity for the coefficients $A_{ijk}$ and $B_{ijk}$ given in \eqref{eq:condition}, to obtain
	\begin{equation}
		\begin{aligned} \label{eq:leftside:EC:barP}
 & \jump{\bB{w}}^{\rm T} \mbf{\bar{f}}^{\tt EC, LM}
 - \avg{\bB{w}^{\rm T} \mbf{\bar{B}}^{\tt LM}} \jump{\bB{u}}
			=
			\sum\limits_{i,j,k=1}^N \tilde{A}_{ijk} \jump{\alpha_i}\avg{h\alpha_j}\avg{\alpha_k} - \tilde{B}_{ijk}\avg{\alpha_i\alpha_k}\jump{h\alpha_j}\\
			&\qquad =
		    \sum\limits_{i,j,k=1}^N -\big(\tilde{B}_{ijk} + \tilde{B}_{kji} \big) \jump{\alpha_i}\avg{h\alpha_j}\avg{\alpha_k} - \tilde{B}_{ijk}\avg{\alpha_i\alpha_k}\jump{h\alpha_j}\\
		    &\qquad =
		    \sum\limits_{i,j,k=1}^N -\tilde{B}_{ijk}
		    \Big(
		    \jump{\alpha_i}\avg{h\alpha_j}\avg{\alpha_k} +
		    \jump{\alpha_k}\avg{h\alpha_j}\avg{\alpha_i} +
		    \avg{\alpha_i\alpha_k}\jump{h\alpha_j}
		    \Big)\\
		    &\qquad =
		    \sum\limits_{i,j,k=1}^N -\tilde{B}_{ijk}\jump{h\alpha_i\alpha_j\alpha_k} = \jumpbig{\,\overline{\mathcal{P}}^{\tt LM}},
		\end{aligned}
	\end{equation}
where $\overline{\mathcal{P}}^{\tt LM}$ is defined in \eqref{eq:potential:LM:bar}. Thus, replacing \eqref{eq:leftside:EC:P} and \eqref{eq:leftside:EC:barP} into \eqref{eq:leftside:EC:linearity} and using Equation~\eqref{eq:potential:hierarchy}, we conclude that
\begin{align*}
 \jump{\bB{w}}^{\rm T} \mbf{f}^{\tt EC} - \avg{\bB{w}^{\rm T} \mbf{B}} \jump{\bB{u}}  = \jump{\mathcal{P}^{\tt LM}} + \jumpbig{\,\overline{\mathcal{P}}^{\tt LM}} = \jump{\bB{w}^{\rm T}\mbf{f} - \mathbb{F}},
\end{align*}
which concludes the proof.
\end{proof}

\subsection{Entropy stable discontinuous Galerkin scheme}\label{sec:es_dgsem}

The result from Theorem~\ref{thm:ec_flux} shows that we can construct an EC scheme if we introduce EC fluctuations \eqref{eq:ec_fluctuation_definition} built from the EC flux \eqref{eq:f:EC} for both volume and surface contributions of the DGSEM \eqref{eq:dgsem}.
While the EC formulation is valid for smooth solutions, for discontinuous solutions entropy should be dissipated.
So, instead, we require an entropy stable (ES) formulation that recovers the entropy inequality \eqref{eq:entropy_inequality} discretely. 
To this end, we extend the EC fluctuation \eqref{eq:ec_fluctuation_definition} to one that is ES by adding a modified version of the Rusanov dissipation, that is formulated in terms of entropy variables in \cite{fjordholm2012energy, ranocha2017shallow}
\begin{equation}\label{eq:es_fluctuation_definition}
	\mbf{D}^{\pm}_{\tt ES} \,=\, \mbf{D}^{\pm}_{\tt EC} \pm \tfrac{1}{2}|\lambda|_{\max}\mbf{Q}_{\tt ES}\jump{\bB{w}} \,=\, \mbf{D}^{\pm}_{\tt EC} \pm \tfrac{1}{2}|\lambda|_{\max} 
    \begin{pmatrix} 
       \smash{\bar{\mbf{H}}} & 0 \\ 
       0 & 0 
    \end{pmatrix}\jump{\bB{w}},
\end{equation}
where $|\lambda|_{\max}$ denotes the largest absolute eigenvalue from \eqref{eq:eigenvalues:SWLME} and 
\begin{equation*}
	\bar{\mbf{H}} := \frac{1}{g}\left(\avg{\mbf{y}}\avg{\mbf{y}}^{\rm T} + \text{diag}(\avg{\mbf{z}}) \right)
    =
    \frac{1}{g}
    \begin{pmatrix}
    1 & \avg{v} & \avg{\alpha_1}  & \dots & \avg{\alpha_N}\\
    \avg{v} & \avg{v}^2 + g\avg{h} & \avg{v}\avg{\alpha_1} & \dots & \avg{v}\avg{\alpha_N} \\
    \avg{\alpha_1} & \avg{\alpha_1}\avg{v} & \avg{\alpha_1}^2 + 3g\avg{h} & \dots & \avg{\alpha_1}\avg{\alpha_N} \\
    \vdots & \vdots & \vdots  & \ddots & \vdots              \\
    \avg{\alpha_N} & \avg{\alpha_N}\avg{v}  & \avg{\alpha_N}\avg{\alpha_1}        & \dots & \avg{\alpha_N}^2 + (2N+1)g\avg{h}
    \end{pmatrix},
\end{equation*}
with vectors $\mbf{y} = (1, u_m, \bB{\alpha}^{\rm T})^{\rm T}$ and $\mbf{z} = (0, gh, 3gh, \dots, (2N+1)gh)^{\rm T}$ is the inverse of the entropy Hessian \eqref{eq:entropy:hessian}, evaluated at some intermediate state $\bar{\bB{u}} \in [\bB{u}^-, \bB{u}^+]$, that for  constant bottom topography satisfies the relation $\bar{\mbf{H}}\jump{(w_1, ..., w_{N+2})^{\rm T}} = \jump{(u_1, ..., u_{N+2})^{\rm T}}$.
\begin{remark}
	While the standard Rusanov dissipation is shown to be entropy dissipative, e.g., in \cite{ersing2026new, ranocha2018comparison}, these results require a spatially convex entropy function which does not apply in the case of non-constant  bottom topography.
\end{remark}

Assuming $h>0$, Lemma~\ref{lemma:spd:hessian} implies that $\bar{\mbf{H}}$ is symmetric positive definite and therefore $\frac{1}{2}|\lambda|_{\max}\jump{\bB{w}}^{\rm T}\mbf{Q}_{\tt ES}\jump{\bB{w}}$ is non-negative.
Then, from \cite[Theorem~4]{ersing2026new} the scheme
\begin{equation}\label{eq:es_dgsem}
\begin{aligned}
	&\omega_i \frac{\Delta x_k}{2} \partial_t \mbf{U}_i^k + \omega_i \sum\limits_{m=0}^P 2\mathcal{D}_{im} \mbf{D}_{\tt EC}^{-}\big(\mbf{U}_i^k, \mbf{U}_m^k\big) + \, 
    \delta_{i0}\mbf{D}_{\tt ES}^{+}\big(\mbf{U}_P^{k-1}, \mbf{U}_0^k\big) + 
    \delta_{iP}\mbf{D}_{\tt ES}^{-}\big(\mbf{U}_P^k, \mbf{U}_0^{k+1}\big) = \mbf{0},
\end{aligned}
\end{equation}
for the degree of freedom $i = 0,1,\dots,P$, with ES surface fluctuations \eqref{eq:es_fluctuation_definition} and EC volume fluctuations \eqref{eq:ec_fluctuation_definition} is ES, i.e., it recovers a semi-discrete integral version of the entropy inequality \eqref{eq:entropy_inequality}.

\subsection{Well-balanced approximation}

Another important continuous property of both the SMWE \eqref{sys:SWME} and SWLME \eqref{sys:SWLME} is the existence of steady states, where the solution remains constant in time.
For both kinds of models, several distinct steady state configurations are known and different approaches such as the global flux method \cite{Ciallela2026}, relaxation methods \cite{Caballero2025} or a formulation in terms of equilibrium variables \cite{FanArxiv2025} have demonstrated preservation for a general class of steady states.
Specifically, one important steady state solution that we consider in this work is the so-called lake-at-rest steady state given by
\begin{equation}\label{eq:lake_at_rest}
	\mathcal{U}_{\text{wb}} = \Big\{\bB{u} \; : \; h + b = H_0, \quad h u_m = 0, \quad h\bB{\alpha} =  \bB{0}, \quad h > 0 \Big\},
\end{equation}
where $H_0$ denotes some constant water level. 
Preservation of this steady state solution in the discrete case, called well-balancing, is an important aspect of any numerical method as many solutions represent small perturbations around this steady state and violations may introduce artificial waves on the order of the grid spacing \cite{chertock2018well}.
The following Lemma demonstrates that our ES numerical scheme \eqref{eq:es_dgsem} exactly satisfies such a well-balanced property at the discrete level.

\begin{lemma}\label{lemma:well_balancedness}
	The ES scheme \eqref{eq:es_dgsem} with fluctuations \eqref{eq:es_fluctuation_definition} and \eqref{eq:ec_fluctuation_definition} constructed from the EC flux \eqref{eq:f:EC} is well-balanced for the lake-at-rest steady state, i.e., it exactly preserves solutions satisfying the lake-at-rest conditions \eqref{eq:lake_at_rest}.
\end{lemma}
\begin{proof}
	We show that both volume and surface fluctuations in \eqref{eq:es_dgsem} vanish point-wise, when evaluated under lake-at-rest conditions \eqref{eq:lake_at_rest}.
	
	First consider the surface fluctuations \eqref{eq:ec_fluctuation_definition}. From the lake-at-rest \eqref{eq:lake_at_rest} most components vanish as momenta and moments are set to zero due to vanishing velocities.
    The remaining contributions vanish as the total height $h+b$ is constant
	\begin{equation*}
		\mbf{D}_{\tt EC}^{\pm}(\bB{u}^-, \bB{u}^+) = 
		\big(0, \tfrac{1}{2}gh^{\pm}\jump{h + b}, 0, \hdots, 0\big)^{\rm T}
		= \mbf{0}, \quad \text{for all }\bB{u}^-, \bB{u}^+ \in \mathcal{U}_{\text{wb}}.
	\end{equation*}	
	Then, to show that the volume fluctuations \eqref{eq:es_fluctuation_definition} vanish, we only have to show that the additional dissipation term vanishes
	\begin{equation*}
		\mbf{D}^{\pm}_{\tt ES}(\bB{u^-}, \bB{u}^+) = \mbf{D}_{\tt EC}^{\pm}(\bB{u}^-, \bB{u}^+) \pm \tfrac{1}{2}|\lambda|_{\max}\begin{pmatrix} 
       \smash{\bar{\mbf{H}}} & 0 \\ 
       0 & 0 
    \end{pmatrix}
    \jump{\bB{w}} = \mbf{0},  \quad \text{for all }\bB{u}^-, \bB{u}^+ \in \mathcal{U}_{\text{wb}},
	\end{equation*}
	which follows directly as the unaugmented entropy variables \eqref{eq:SWME:Entropy:Variables} assume a constant value for the lake-at-rest conditions \eqref{eq:lake_at_rest}, while the last entry is multiplied by zero. 
	Substituting both results into \eqref{eq:es_dgsem} shows that the time-derivative vanishes.
\end{proof}

\section{Numerical examples} \label{sec:numerical:examples}

In this section, we present a set of numerical examples to verify our theoretical findings and demonstrate the performance of the proposed ES numerical scheme from Section \ref{sec:numerical:scheme}. The numerical examples have been produced making use of the open-source framework Trixi.jl \cite{Ranocha2022,Schlottke2021} for the semi-discretization. For the time integration, we employ a low storage five-stage fourth-order Runge--Kutta scheme by Carpenter and Kennedy \cite{Carpenter1994} as implemented in DifferentialEquations.jl \cite{rackauckas2017differentialequations}. Unless the time step is specified, we use a CFL-based time step using the eigenvalues of the SWLME given by \eqref{eq:eigenvalues:SWLME}, which for the case of the SWME, only corresponds to an estimation.

To cover the cases in which the solution may feature shocks or high gradients, we supplement our ES scheme \eqref{eq:es_dgsem} in Section~\ref{sec:numerical:scheme} with the subcell shock capturing approach developed in \cite{Hennemann2021} using the indicator variable $u_m^3$, that was selected heuristically as it showed favourable shock-capturing properties in our numerical tests. The shock capturing method is active in all examples with the exception of the accuracy test in Example~\ref{ex:3}, as it may affect the spatial order of accuracy.

The necessary code and instructions to reproduce the numerical results in this section are provided in a reproducibility repository \cite{ersing2026swmeRepro}.

\subsection{Example 1: Smooth wave with friction effect}\label{ex:1}

With the purpose of comparing our numerical scheme with existing solvers for SWME, we consider the standard test example of a travelling water wave with periodic boundary conditions described in \cite[Section 5]{Kowalski2019}. In this example, the friction is modelled via the Newtonian slip law given by the function $\mbf{S}=\mbf{S}_{\rm Ns}$ \eqref{eq:friction:Nslip}, which corresponds to the source term in \eqref{sys:SWME:vectorial}. Furthermore, we set the slip length to $\lambda = 0.1$ and the friction constant to $\nu = 0.1$, and use $g=1$. The water body is initially described by a smooth bump given by the following height function
\begin{align*}
 h(x,0) = 1 + {\rm exp}\Big(3\cos\big(\pi (x + \tfrac{1}{2} )\big) - 4\Big)\qquad \text{for } x\in [-1,1],
\end{align*}
while the bathymetry function for the flat bottom corresponds to $b(x)=0$. The initial average velocity and moments, for the modelling polynomial order $N=2$, are given by
\begin{align*}
u_m(x,0)      = 0.25,\qquad 
\alpha_1(x,0) = 0 \quad \text{and}\quad 
\alpha_2(x,0) = -0.25\qquad \text{for } x\in[-1, 1].
\end{align*}
For the reference approximate solution, we use a finite volume polynomial viscosity based scheme (PVM) described in \cite{CareagaArxiv2025}, in which system \eqref{sys:SWME:vectorial} is written in a fully non-conservative formulation combining the Jacobian of $\mbf{f}$ and matrix $\mbf{B}$ into a single system matrix. In addition, the viscosity matrix used is given by the HLL method, the path integrals are computed with linear paths and third order quadrature rule, time approximation is via forward Euler, and the Fortran 90 routines used can be found in \cite{careaga_2026_18468724}.

\begin{figure}[t]
\centering
\includegraphics[width = 0.95\textwidth]{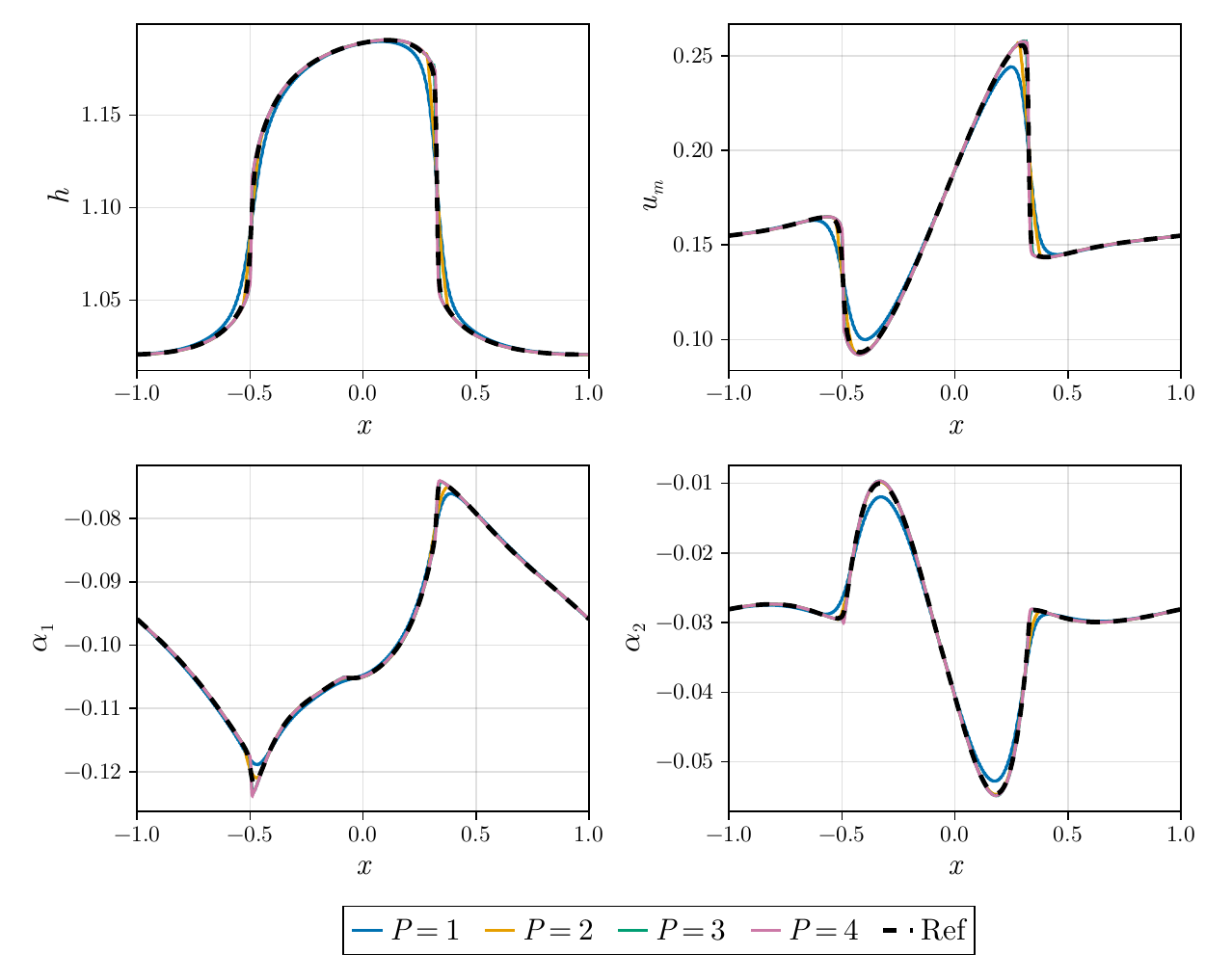}
 \caption{Example 1: Numerical solutions (in primitive variables) of the SWME \eqref{sys:SWME} at $t=2$ with $N = 2$ and source term given by \eqref{eq:friction:Nslip} with $\lambda = 0.1$ and $\nu = 0.1$, for different polynomial degrees $P$. In all simulations $K=256$ and $\text{CFL = 0.9}$. \label{fig:example1a}}
\end{figure}

In Figure~\ref{fig:example1a}, we present snapshots of the numerical solution at time $t=2$ computed with the DG entropy preserving numerical scheme from Section~\ref{sec:numerical:scheme} on a mesh with $K=256$ elements varying the polynomial degree $P$. The reference solution, which is approximated with $2500$ cells, corresponds to the black dashed line. In all components of the solution, the DG approximations tend to the reference solution as the polynomial degree $P$ increases. We observe that already for $P=4$, the numerical solution is comparable to the reference solution, even when the difference in the number of layers is large. 
Now, to visually compare our approximate solutions upon mesh refinement, we fix the polynomial degree to $P=2$ and compute the solutions for $K=128,256,512$. As before, all the components of the approximated solutions are compared to the reference curve in Figure~\ref{fig:example1b} at time $t=2$. As expected, the approximations converge to the reference curve, where the finer refinement for $K=512$ closely matches the reference solution. 

\begin{figure}[t]
\includegraphics[width = 0.95\textwidth]{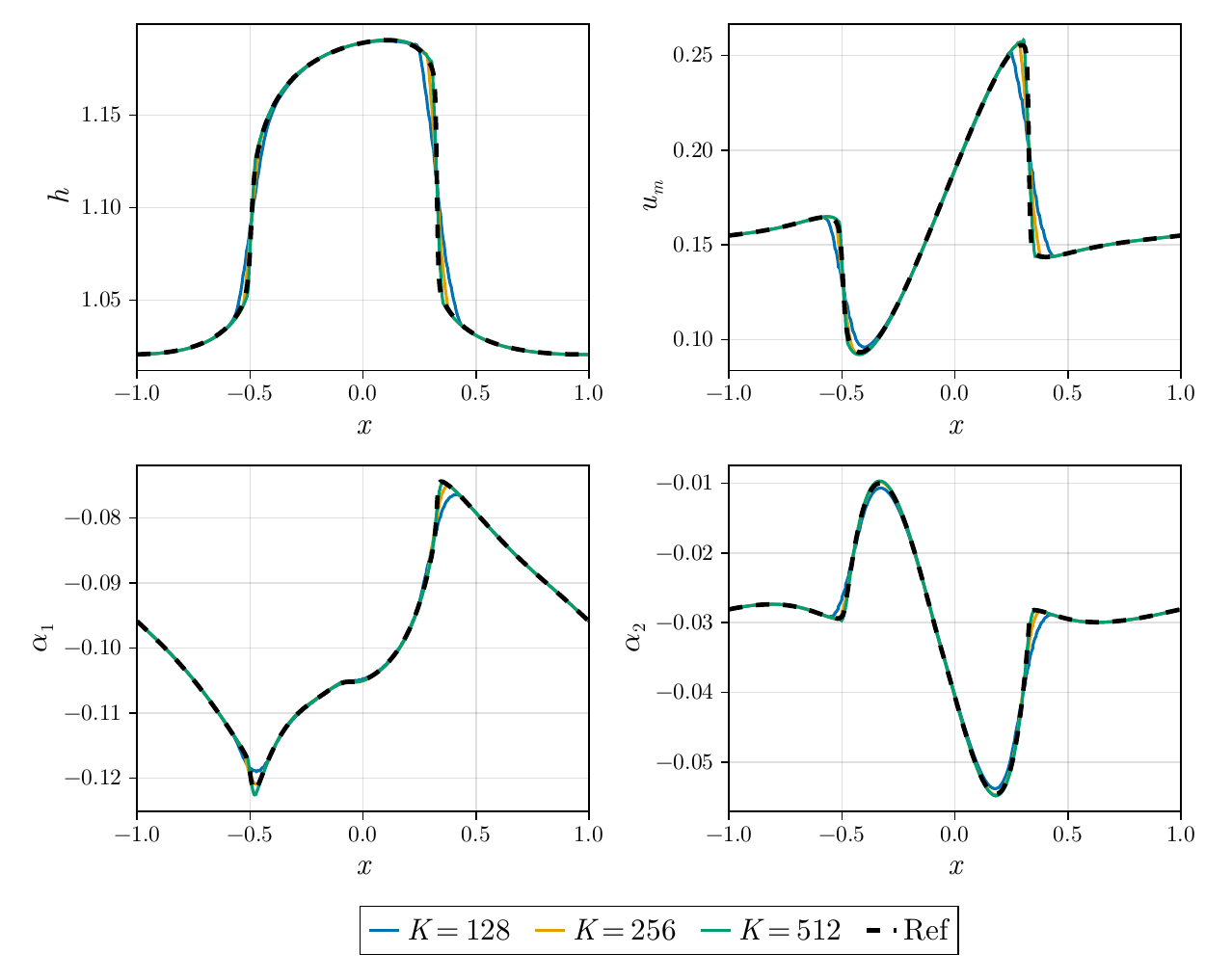}
\caption{Example 1: Numerical solutions (in primitive variables) of the SWME \eqref{sys:SWME} at $t=2$ with $N = 2$ and source term given by \eqref{eq:friction:Nslip} with $\lambda = 0.1$ and $\nu = 0.1$, for different numbers of elements $K$. In all simulations, the polynomial degree is $P = 2$ and $\text{CFL = 0.9}$. \label{fig:example1b}}
\end{figure}

\subsection{Example 2: Entropy dissipation of the friction terms}\label{ex:2}

We now study the discrete entropy dissipation due to the Newtonian slip friction $\mbf{S} = \mbf{S}_{\rm Ns}$ \eqref{eq:friction:Nslip} and Newtonian Manning law $\mbf{S} =\mbf{S}_{\rm NM}$ \eqref{eq:friction:NManning} described in Section~\ref{sec:Entropy:dissipation}. For this, we consider the same initial condition as in Example 1, with actual gravitational acceleration constant $g=9.81$, and for $\mbf{S}_{\rm NM}$ we use the Manning coefficient $n=0.0165$ and constant density $\rho = 1000$. Furthermore, we use here the non-homogeneous version of the linearized model SWLME defined in \eqref{sys:SWLME}. To measure the entropy dissipation in the entire domain $\Omega = [-1,1]$ at each time point, we define the following quantities 
\begin{align*}
\mathcal{D}_{\rm Ns}(t) = \frac{1}{|\Omega|} \int_{\Omega} \bB{w}^{\rm T}\mbf{S}_{\text{Ns}}(\bB{u})\,\text{d}x,
\quad 
\mathcal{D}_{\rm NM}(t) = \frac{1}{|\Omega|} \int_{\Omega} \bB{w}^{\rm T}\mbf{S}_{\text{NM}}(\bB{u})\,\text{d}x, 
\end{align*}
which are evaluated numerically using the polynomial approximation and quadrature formulas as described in Section~\ref{sec:numerical:scheme}.
In Figure~\ref{fig:example2}, we show the entropy dissipation at each time for the Newtonian slip $\mathcal{D}_{\rm Ns}$, and Newtonian Manning case $\mathcal{D}_{\rm NM}$ varying the friction parameter $\nu$. As shown in Section~\ref{sec:Entropy:dissipation}, in all cases and for both friction terms, the entropy dissipation remains negative and tends to zero as the time evolves. In both friction terms, the effect of larger values of $\nu$ is more pronounced at early times, where the entropy dissipation curves for $\nu=1$ are the ones of larger magnitude, with minima at $t=0$. The small oscillations in both cases can be explained as a consequence of the periodic boundary conditions.

\begin{figure}[t]
\centering
\includegraphics[width=0.9\textwidth]{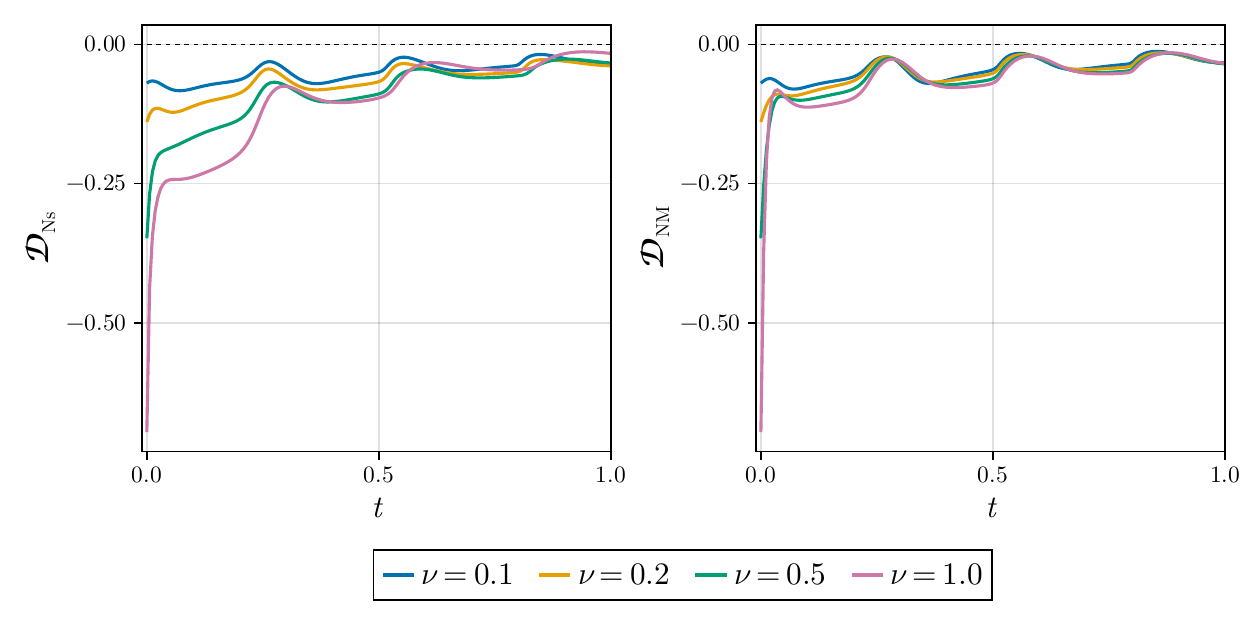} 
\caption{Example 2: Entropy dissipation for the Newtonian slip friction $\mbf{S}_{\rm Ns}$ \eqref{eq:friction:Nslip} and Newtonian Manning friction $\mbf{S}_{\rm NM}$ \eqref{eq:friction:NManning} with respect to the entropy variables $\bB{w}$ defined in Corollary~\ref{cor:entropy:variables}. In all simulations $K=256$ and $\text{CFL = 0.9}$.\label{fig:example2}}
\end{figure}

\setlength{\tabcolsep}{6pt} 

\begin{table}[t!]
	\renewcommand{\arraystretch}{1.1} 
	\centering
	\begin{tabular}{ c|cc|cc|cc|cc|cc}
		\hline
		SWME & 
		\multicolumn{2}{c|}{$h$} & \multicolumn{2}{c|}{$h u_{m}$} & \multicolumn{2}{c|}{$h \alpha_1$} & \multicolumn{2}{c|}{$h \alpha_2$} & \multicolumn{2}{c}{$b$} \\
		\hline
		$K$ & $e_{1,\kappa}$ & $\theta_{1,\kappa}$ &  $e_{2,\kappa}$ & $\theta_{2,\kappa}$ &  $e_{3,\kappa}$ & $\theta_{3,\kappa}$ &  $e_{4,\kappa}$ & $\theta_{4,\kappa}$ &  $e_{5,\kappa}$ & $\theta_{5,\kappa}$\\[0.5ex]
		\hline
		$2^6 $ & 4.08e-07 & - & 3.44e-06 & - & 2.84e-07 & - & 3.52e-07 & - & 1.09e-08 & - \\
		$2^7 $ & 2.56e-08 & 3.99 & 2.15e-07 & 4.00 & 8.16e-09 & 5.12 & 1.04e-08 & 5.07 & 6.82e-10 & 4.00 \\
		$2^8 $ & 1.59e-09 & 4.00 & 1.34e-08 & 4.00 & 1.02e-09 & 3.00 & 7.71e-10 & 3.76 & 4.26e-11 & 4.00 \\
		$2^9 $ & 9.97e-11 & 4.00 & 8.39e-10 & 4.00 & 5.44e-11 & 4.23 & 4.45e-11 & 4.12 & 2.66e-12 & 4.00 \\
		$2^{10} $ & 6.23e-12 & 4.00 & 5.25e-11 & 4.00 & 3.21e-12 & 4.08 & 3.02e-12 & 3.88 & 1.66e-13 & 4.00 \\
		\hline 
        \multicolumn{11}{c}{}\\[-1.6ex]
	\hline SWLME & 
	\multicolumn{2}{c|}{$h$} & \multicolumn{2}{c|}{$h u_{m}$} & \multicolumn{2}{c|}{$h \alpha_1$} & \multicolumn{2}{c|}{$h \alpha_2$} & \multicolumn{2}{c}{$b$} \\
	\hline
	$K$ & $e_{1,\kappa}$ & $\theta_{1,\kappa}$ &  $e_{2,\kappa}$ & $\theta_{2,\kappa}$ &  $e_{3,\kappa}$ & $\theta_{3,\kappa}$ &  $e_{4,\kappa}$ & $\theta_{4,\kappa}$ &  $e_{5,\kappa}$ & $\theta_{5,\kappa}$\\[0.5ex]
	\hline
	$2^6 $ & 3.97e-07 & - & 3.44e-06 & - & 9.43e-07 & - & 9.43e-07 & - & 1.09e-08 & - \\
	$2^7 $ & 2.49e-08 & 4.00 & 2.15e-07 & 4.00 & 5.57e-08 & 4.08 & 5.57e-08 & 4.08 & 6.82e-10 & 4.00 \\
	$2^8 $ & 1.56e-09 & 3.99 & 1.34e-08 & 4.00 & 3.11e-09 & 4.16 & 3.11e-09 & 4.16 & 4.26e-11 & 4.00 \\
	$2^9 $ & 9.86e-11 & 3.99 & 8.40e-10 & 4.00 & 1.56e-10 & 4.32 & 1.56e-10 & 4.32 & 2.66e-12 & 4.00 \\
	$2^{10} $ & 6.22e-12 & 3.99 & 5.25e-11 & 4.00 & 6.89e-12 & 4.50 & 6.89e-12 & 4.50 & 1.66e-13 & 4.00 \\
	\hline
\end{tabular}
\captionsetup{skip=10pt}
\caption{Example 3: Accuracy tests in the $L^2$-norm for the SWME (top table) and SWLME (bottom table), with $N=2$ and polynomial degree $P=3$ at $t=0.05$. The errors $e_{i,\kappa}$ and rates $\theta_{i,\kappa}$ for $i=1,...,5$ are defined in \eqref{eq:L2errors} and \eqref{eq:convergence:rates}, respectively. Results were obtained with a fixed time step $\Delta t = 10^{-5}$.
\label{tab:accuracy}}
\end{table}

\subsection{Example 3: Accuracy tests} \label{ex:3}

In order to determine the order of accuracy produced by our developed numerical scheme, we use as exact solution a manufactured solution, for which an additional source term needs to be supplemented to system \eqref{sys:SWME:vectorial}. We let $\bB{u}_{\rm ex} = \big(h_{\rm ex},h_{\rm ex}u_{m, \rm ex},h_{\rm ex}\bB{\alpha}_{\rm ex}^{\rm T}, b\big)^{\rm T}$ be the smooth manufactured exact solution defined component-wise by
\begin{align*}
 &b(x) = 2 + \tfrac{1}{2} \sin\big(\sqrt{2}\pi x\big),\qquad h_{\rm ex}(x,t)    = \Big(7 + \cos\big(2\sqrt{2}\pi x\big) \cos\big(2\pi t\big)\Big) - b(x),\\
&u_{m,\rm ex}(x,t) = \tfrac{1}{2}, \qquad \alpha_{i,{\rm ex}}(x,t)  = \tfrac{1}{2}\qquad \text{for }i = 1,\dots,N,
\end{align*}
for all $x\in \big[0,\sqrt{2}\,\big] =\Omega$ and time $t\geq 0$. The corresponding source term related to the chosen manufactured solution is then obtained by replacing $\bB{u}_{\rm ex}$ into the left-hand side of \eqref{sys:SWME:vectorial}, that is
\begin{align*}
 \mbf{S}_{\rm ex}(x,t) = \partial_t\bB{u}_{\rm ex} + \partial_x \mbf{f}(\bB{u}_{\rm ex}) + \mbf{B}(\bB{u}_{\rm ex})\partial_x\bB{u}_{\rm ex}.
\end{align*}
We let $\bB{u}_{\kappa}^n$ be the discrete solution computed with our numerical scheme at time $t^n$ with a fixed meshsize $\kappa = \Delta x = \sqrt{2}/2^l$ and $l\in \mathbb{N}$, where the number of elements is $K= 2^l$. Then, to compute the numerical errors due to the spatial approximations, we compute the following $L^2$-error of each of the unknown at a fixed time $t^n$
\begin{align} \label{eq:L2errors}
e_{i,\kappa}(t^n) &= \|\left(\bB{u}^n_{\kappa} - \bB{u}_{\rm ex}(t^n)\right)_i\|_{L^2(\Omega)},
\end{align}
with $i = 1,\dots,N+3$, for a sequence of mesh refinements of $K = 2^{6},...,2^{10}$ elements. Note that since the bathymetry function $b$ is constant in time, the numerical error related to this component is given by its $L^2$-projection into the discontinuous Galerkin space. The rate of convergence between two consecutive values $\tilde{\kappa} = \sqrt{2}/2^{l-1}$ and $\kappa = \sqrt{2}/2^{l}$ is defined as
\begin{align} \label{eq:convergence:rates}
\theta_{j,\kappa}(t^n) = {\rm log}(e_{j,\tilde{\kappa}}^n/e_{j,\kappa}^n)/ {\rm log}(2),\qquad \text{for }j=1,\dots, N+3.
\end{align}
Furthermore, we fix the modelling order to $N=2$ and use a fixed time step on each set of examples. In Table~\ref{tab:accuracy}, we present the numerical errors and rates of convergence for the SWME \eqref{sys:SWME} and SWLME \eqref{sys:SWLME} for polynomial degree $P=3$ and $\Delta t=10^{-5}$, at final time $t=0.05$. For both models, the results show that the numerical errors decrease with respect to the mesh refinement and the orders of convergence, for most of the components, tend to approximately $P+1$, which is the expected order of convergence for a nodal discontinuous Galerkin approximation. Furthermore, we observe that although the approximation related to the computation of $A_{ijk}$ and $B_{ijk}$ is not in effect for the SWLME (bottom table), the errors related to $h$ and $h u_m$ do not differ significantly to the SWME (top table). In the particular case of the moment components $h\alpha_1$ and $h\alpha_2$, the errors in the SWLME are slightly larger than those for the SWME; however, the convergence rates are also higher.

\begin{figure}[t!]
	\centering
	\includegraphics[width=0.9\textwidth]{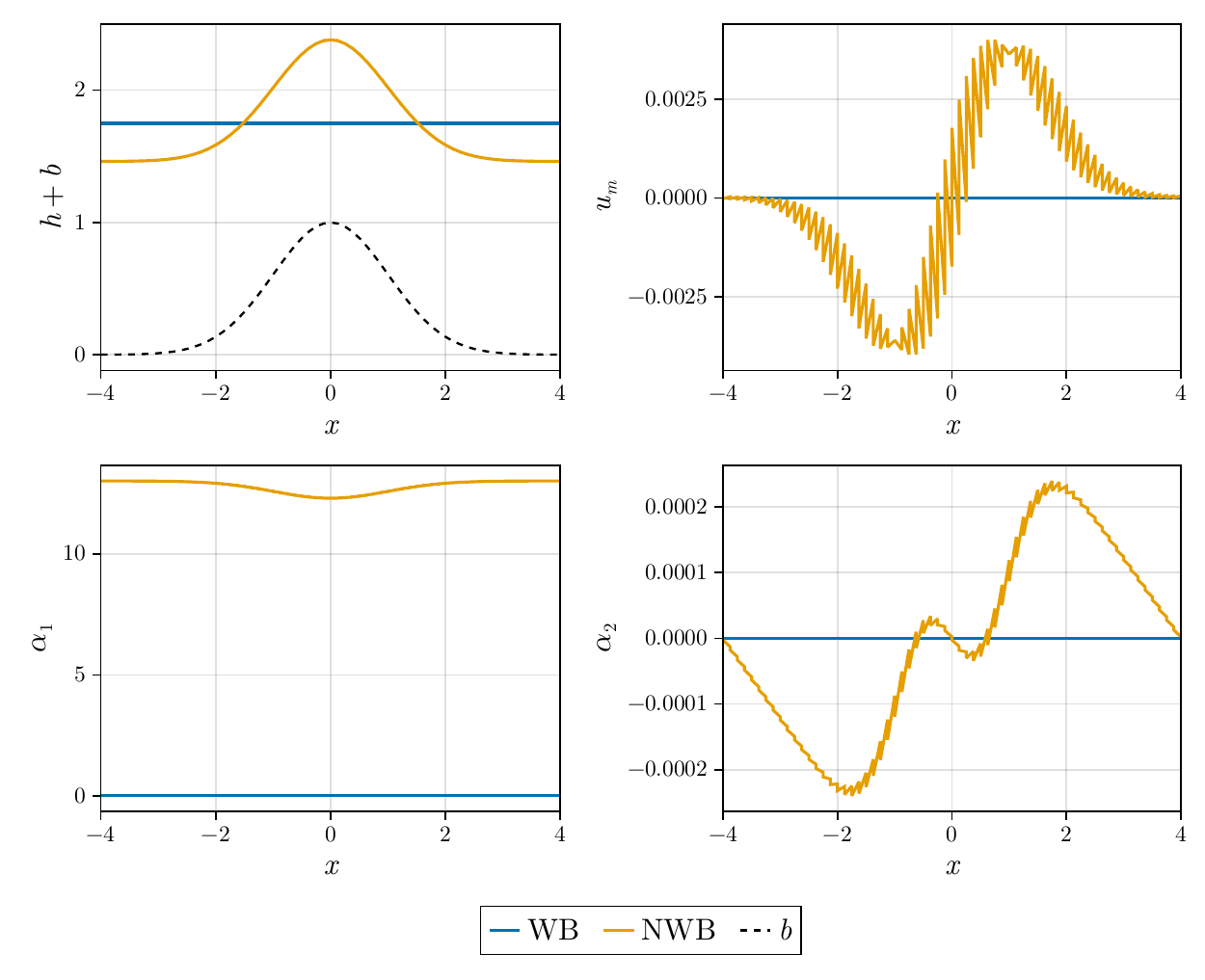} 
	\caption{Example 4: Numerical solutions of the SWME \eqref{sys:SWME} for the lake-at-rest test case at $t=8000$ using the well-balanced (WB) and the non well-balanced (NWB) scheme. Results are obtained on $K=64$ elements with polynomial degree $P=1$ and $\text{CFL}=0.9$. \label{fig:example4}}
\end{figure}

\subsection{Example 4: Lake-at-rest test}\label{ex:4}

We consider a test case to verify numerically that the scheme proposed in \eqref{eq:es_dgsem} preserves the lake-at-rest steady state \eqref{eq:lake_at_rest} for the SWME as shown in Lemma~\ref{lemma:well_balancedness}. 
For this we consider the domain $\Omega = [-4, 4]$ with periodic boundary conditions on which we prescribe the following initial conditions
\begin{equation}\label{eq:IC:example4}
	\begin{aligned}
	H(x,0) &= 1.75, \quad b(x) = {\rm e}^{- x^2 / 2}, \quad h(x,0) = H(x,0) - b(x), \\
	u_m(x,0) &= \alpha_1(x,0) = \alpha_2 (x,0) = 
	\begin{cases}
		-10^{-3} & \text{for } -1 \leq x < 0, \\
		  10^{-3} & \text{for }  0 \leq x \leq 1,\\
		     0 & \text{for }  1 < |x|,
	\end{cases}
	\end{aligned}
\end{equation}
for all $x\in \Omega$. The functions in~\eqref{eq:IC:example4} represent a minor perturbation in the velocity and moments around a lake-at-rest with smooth bottom topography. 
We further set the gravitational acceleration $g=9.812$ and consider the SWME with $N=2$ moments. 
To obtain the numerical results, we discretize the domain into $64$ equidistant elements and compute the solution from $t_0 = 0$ up to a final time $t_{\rm end}=8000$ to examine long time behaviour.

In Figure~\ref{fig:example4}, we present numerical results for polynomial degree $P=1$ at final time, which are obtained with our proposed ES scheme \eqref{eq:es_dgsem} and a non well-balanced scheme, where we replaced the interface dissipation with the standard Rusanov dissipation term. Additionally, in Figure~\ref{fig:example41}, we present time series data for the total entropy and the lake-at-rest error in the water height for the same configuration.
\begin{figure}[t!]
	\centering
	\includegraphics[width=0.9\textwidth]{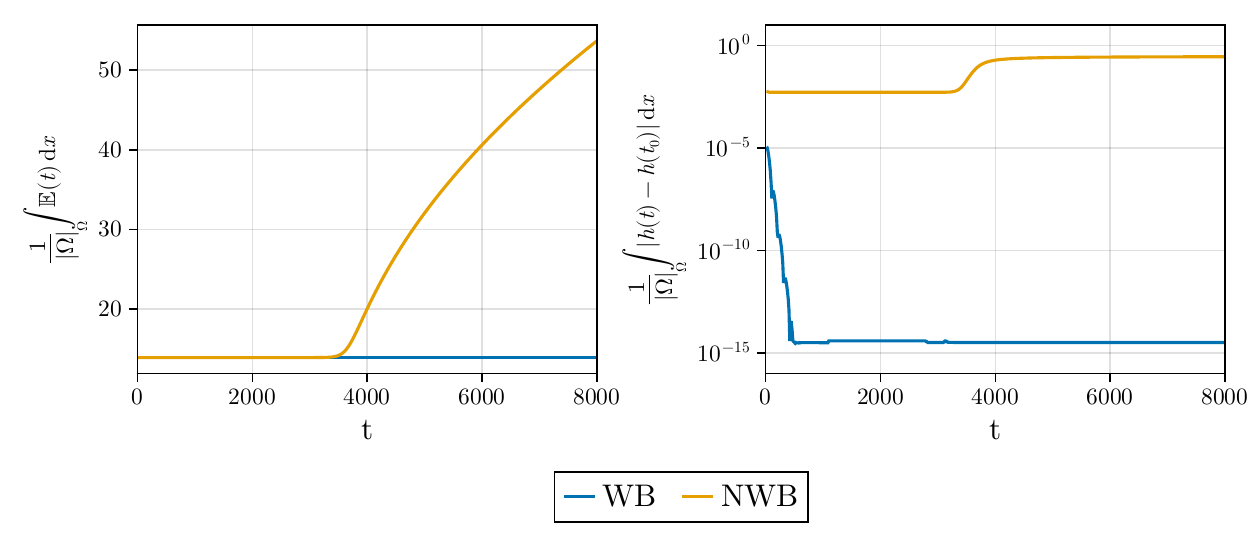} 
	\caption{Example 4: Time series data for the total entropy and the lake-at-rest error for the lake-at-rest test case using the well-balanced (WB) and the non well-balanced (NWB) scheme. Results are obtained on $K=64$ elements with polynomial degree $P=1$ and $\text{CFL}=0.9$. \label{fig:example41}}
\end{figure} 

For the ES scheme, we observe that over time the initial perturbation in the velocity profile gets dissipated by the numerical dissipation and even after considerable time we recover the original steady state configuration up to machine accuracy with a maximum node-wise lake-at-rest error of $|h(t_{\rm end}) - h(t_0)| \approx 3.414\cdot 10^{-15}$, which demonstrates the well-balanced property of the scheme. For the non well-balanced scheme, on the other hand, we initially observe oscillations in the water height as is typical when using a non well-balanced scheme for the SWE. In Figure~\ref{fig:example41} this causes the first sudden rise in the lake-at-rest error, that then remains almost constant over considerable time.
However, the initial moment perturbations then keep growing over time until the solution starts to converge towards a completely different solution, which becomes evident through a noticeable rise in both total entropy and lake-at-rest error shortly before $t = 4000$. At later times the solution then seemingly approaches a different steady state that is depicted in Figure~\ref{fig:example4}.

The results verify the well-balanced property for our ES scheme as established in Lemma~\ref{lemma:well_balancedness}. Furthermore, the behaviour for the non well-balanced scheme shows how violations of the entropy inequality and the well-balanced property may cause artificial oscillations and non-physical solutions.

\section{Conclusions} \label{sec:conclusions}

We developed energy equations for the general shallow water moment equations (SWME) and the linearized shallow water moment equations (SWLME), and established the respective entropy variables. Notably, the derived energy function and entropy variables coincide for both models. We use the derived entropy variables to show that two of the standard models of friction used in moment models, the Newtonian slip and Newtonian Manning, are entropy dissipative.

The results of the continuous entropy analysis were then used to construct an entropy stable and well-balanced discontinuous Galerkin spectral element method (DGSEM) that recovers the energy equations for the SWME in the semi-discrete case.
For the spatial approximation, an augmented system with trivial evolution equation for the bottom topography was introduced to incorporate the bottom topography source term into the nonconservative product. The proposed DGSEM was written in fluctuation form to account for the nonconservative products and a special flux-differencing volume integral in combination with entropy conservative numerical fluxes was introduced to ensure entropy conservation.

For the derivation of these entropy conservative numerical fluxes, a key factor was the hierarchical character of the SWME, which allowed us to derive the entropy conservative fluxes sequentially from those known for the classical shallow water equations \cite{Ersing2024}. A drawback of our constructed entropy conservative flux is that it is specifically tailored for path-integrals (related to non-conservative terms) approximated through trapezoidal rule in combination with linear paths.

An entropy stable variant of the proposed DGSEM, was then obtained by adding a modified Rusanov dissipation term, written in terms of entropy variables, to the fluctuations at element interfaces. We further demonstrated that the resulting method exactly preserves the lake-at-rest steady state. A shock capturing method in the line of \cite{Hennemann2021} is incorporated as an additional functionality to improve the robustness in the presence of discontinuities.
Finally, the theoretical findings and the performance of the numerical method were demonstrated in a series of numerical test cases. 

Extensions of this work can be conducted in a variety of research directions. A direct generalization of the total energy of the system and entropy variables is to consider domains in two horizontal spatial dimensions. The models here presented can also be extended to the non-hydrostatic case or coupled to the Exner equations for sediment transport. A detailed study regarding entropy dissipative models of friction for granular flows is also of interest. With respect to the numerical approximation, further studies can be conducted regarding the approximation of the path integrals related to the nonconservative terms and the corresponding construction of the entropy fluxes. 
More elaborated schemes, including extensions to positivity-preserving schemes that are well-balanced for general steady-state solutions are also of relevance.

\section*{CRediT authorship contribution statement}
\noindent {\bf J. Careaga:} Conceptualization, Methodology, Investigation, Software, Visualization, Writing – Original Draft.\\
{\bf P. Ersing:} Conceptualization, 
Methodology, 
Investigation, 
Software,  
Visualization, 
Writing – Original Draft.\\
{\bf J. Koellermeier:} Conceptualization, 
Methodology, 
Investigation,
Funding Acquisition,
Writing – Review \& Editing.
{\bf A.R. Winters:}
Conceptualization, 
Methodology,
Funding Acquisition,
Software,
Supervision,
Writing – Review \& Editing.

\section*{Declaration of competing interest}

The authors declare that they have no known competing financial interests or personal relationships that could have appeared to influence the work reported in this paper.

\section*{Acknowledgments}

J. Careaga and J. Koellermeier are supported by the Dutch Research Council (NWO) through the ENW Vidi project
HiWAVE with file number VI.Vidi.233.066. P. Ersing and A.R. Winters were supported by Vetenskapsrådet, Sweden [Grant agreement 2020-03642 VR].

\appendix
\section{Proof of Equation~\ref{eq:condition}}\label{sec:appendix:A}
\label{sec:proof}

\begin{proof}
We proceed to prove that $\widetilde{B}_{ijk} + \widetilde{A}_{kji} + \widetilde{B}_{kji} = 0$ for all $i,j,k = 1,\dots,N$ as follows. We first recall that both $A_{i j k}$ and $B_{i j k}$ are defined in \eqref{eq:Aijk_Bijk} and that $A_{i j k} = (2i+1) \widetilde{A}_{i j k}$ and $B_{i j k} = (2i+1) \widetilde{B}_{i j k}$, respectively. Then, we have
\begin{align}
\begin{aligned}
    \widetilde{B}_{ijk} +  \widetilde{A}_{kji} + \widetilde{B}_{kji}
    =& \int_0^1 \phi_i^{\prime}\left(\int_0^\zeta \phi_j(s) \, {\rm d} s\right) \phi_k \, {\rm d} \zeta + \int_0^1 \phi_k \phi_j \phi_i \, {\rm d} \zeta
    \, + \int_0^1 \phi_k^{\prime}\left(\int_0^\zeta \phi_j(s) \, {\rm d} s\right) \phi_i \, {\rm d} \zeta \\
    =&\, \int_0^1 \Big(\phi_i^{\prime} \phi_k \Phi_j + \phi_k^{\prime} \phi_i \Phi_j + \phi_i \phi_j \phi_k \Big)\, {\rm d}\zeta, \label{eq:proof1}
\end{aligned}
\end{align}
where $\Phi_j(\zeta) = \displaystyle \int_{0}^{\zeta} \phi_j(s) \, {\rm d} s$. Now, we use integration by parts for the last term in \eqref{eq:proof1} to obtain
\begin{align}
    \int_0^1  \phi_k \phi_j \phi_i \, {\rm d}\zeta  = \Big[ \phi_k(\zeta) \phi_i(\zeta) \Phi_j(\zeta)\Big]_0^1 - \int_0^1 \phi_k \phi_i^{\prime}\Phi_j \, {\rm d} \zeta - \int_0^1 \phi_k^{\prime} \phi_i\Phi_j \, {\rm d} \zeta.
\end{align}
Inserting this into \eqref{eq:proof1} finally yields
\begin{align}
    \int_0^1 \Big(\phi_i^{\prime} \phi_k \Phi_j + \phi_k^{\prime} \phi_i \Phi_j + \phi_i \phi_j \phi_k\Big) \ {\rm d}\zeta = \Big[ \phi_k(\zeta) \phi_i(\zeta) \Phi_j(\zeta)\Big]_{0}^{1} = 0,
\end{align}
where we made use of $\Phi_j(0) = 0$ and $\Phi_j(1) = 0$ for all $j =1,\dots,N$.
\end{proof}

\section{Computation of moment model components}\label{sec:computeTensors}

The moment model introduces two tensors and one matrix built from the shifted Legendre basis functions.
The shifted Legendre polynomials can be created from the Rodrigues' formula
\begin{equation*}
\phi_j(\zeta) = \frac{1}{j!}\frac{{\rm d}^j}{{\rm d}\zeta}^j\left(\zeta - \zeta^2\right)^j, \qquad \text{for } j = 1,\ldots,N,
\end{equation*}
where $\zeta\in[0,1]$.
However, for the purpose of computing quantities like $B_{ijk}$, it is convenient to exploit that the Legendre polynomials and their derivatives can be constructed from recurrence relations.
In particular, the shifted Legendre basis comes from the three term recurrence
\begin{equation*}
\phi_0 = 1,\quad
\phi_1 = 1 - 2\zeta,\quad
\phi_j = \tfrac{2j-1}{j}(1 - 2\zeta) \phi_{j-1} - \tfrac{j-1}{j}\phi_{j-2},\qquad \text{for } j = 2, ..., N,
\end{equation*}
and its derivatives come from the recurrence
\begin{equation*}
\phi'_0 = 0,\quad
\phi'_1 = -2,\quad
\phi'_j = \phi'_{j-2} - 2 (2j-1) \phi_{j-1},\qquad \text{for }j = 2, ..., N.
\end{equation*}

We first create the $\bB{C}$ matrix \eqref{eq:CMatrix} needed to include friction in the moment model.
We use the recurrence relation for the shifted Legendre polynomial derivatives to build the necessary components of the integral.
The terms in the integrand $\phi'_i$ and $\phi'_j$ are polynomials of degree at most $N-1$.
So, the highest degree of the integrand is a polynomial of degree $(N-1) + (N-1) = 2N - 2$. 
We compute the integrals for the matrix $C_{ij}$ with $i,j=1,\ldots,N$, with Legendre-Gauss (LG) quadrature.
This quadrature is defined on the interval $[-1,1]$, so we use the affine map
\[
\zeta = \tfrac{1}{2}(\xi + 1),
\]
where $\zeta\in[0,1]$ and $\xi\in[-1,1]$.
Then the integral becomes
 \begin{equation*}
    {C}_{ij} = (2i+1)\int_0^1\phi_i'(\zeta)\phi_j'(\zeta)\,{\rm d}\zeta = \frac{2i+1}{2}\int_{-1}^1\phi_i'(\xi)\phi_j'(\xi)\,{\rm d}\xi,
 \end{equation*}
where $i,j=1,\dots,N$.
By design, LG quadrature with $M+1$ quadrature points is exact for polynomials up to degree $2M+1$.
Thus, we select $M = \lceil(2N - 3) / 2 + 1\rceil$ nodes to guarantee that the quadrature is exact for the construction of the ${\bB{C}}$ matrix.

Next, we create the $B_{ijk}$ tensor \eqref{eq:Aijk_Bijk} for the moment equations.
We use the recurrence relations for the shifted Legendre polynomials and their derivatives to build the necessary components of the integral.
In the integrand, we have $\phi'_i$ which is a polynomial of degree at most $N-1$ and $\phi_j$ and $\phi_k$ are polynomials of degree $N$.
Therefore, $\int_0^{\zeta} \phi_j(s)\,{\rm d}s$ is a polynomial of degree $N+1$ and the integrand for $B_{ijk}$ is, at most, a polynomial of degree $(N-1) + (N+1) + N = 3N$.
We, again, compute the integral with LG quadrature where we select $M = \lceil(3N - 1) / 2  + 1\rceil$ nodes to ensure that the quadrature will be exact for the construction.
As before, we use an affine map on the outer integral to have
\begin{align*}
B_{i j k} 
= \frac{2 i+1}{2} \int_{-1}^1 \phi_i^{\prime}(\xi)\left(\int_0^{\zeta_m} \phi_j(s)\,{\rm d}s\right) \phi_k(\xi)\,{\rm d}\xi,
\end{align*}
where
\begin{align*}
\zeta_m = \tfrac{1}{2}(\xi_m + 1),\qquad \text{for } m = 0,\ldots,M.
\end{align*}
To compute the inner integral we apply another affine transformation
\begin{align*}
s = \tfrac{1}{4}(\xi_m + 1)(\eta + 1),\qquad\text{where } \eta\in[-1,1].
\end{align*}
to have
\begin{align*}
\int_0^{\zeta_m} \phi_j(s)\,{\rm d}s
    = \tfrac{1}{4}(\xi_m + 1)\int_{-1}^1 \phi_j\left(\tfrac{1}{4}(\xi_m + 1)(\eta + 1)\right)\,{\rm d}\eta.
\end{align*}
We then apply LG quadrature to this inner integral with the same $M+1$ points as the outer integral.
This choice of $M$ for the interior integral is overkill in terms of quadrature accuracy, as fewer LG nodes are required for exact integration of this polynomial; however, this choice avoids the need to compute an additional set of quadrature nodes and weights.

Finally, we compute the tensor $A_{ijk}$ \eqref{eq:Aijk_Bijk}, which is the product of three Legendre polynomials.
There exists a closed form for this product due to Adams~\cite{Adams1878}, but it is computationally more efficient to create the $A_{ijk}$ tensor from the already computed $B_{ijk}$ tensor using the relation \eqref{eq:condition}, i.e.,
\begin{align}
     \widetilde{A}_{kji} = -(\widetilde{B}_{ijk} +\widetilde{B}_{kji}),
\end{align}
with $i,j,k = 1,\ldots,N$.

\section{Symmetric positive definiteness of the entropy Hessian}\label{sec:appendix:C}

\begin{lemma}\label{lemma:spd:hessian}
    The Hessian  $\mbf{H}^{-1}$ of the entropy function \eqref{eq:entropy:hessian} with respect to the conservative variables $\mbf{u}$ is symmetric positive definite for $h > 0$.
\end{lemma}
\begin{proof}
Since the matrix is symmetric, positive definiteness follows from Sylvester's criterion if and only if all leading principal minor submatrices have a positive determinant. 
To demonstrate this, first consider the following block decomposition of the entropy Hessian:
\begin{equation}
\mbf{H}^{-1} = \frac{1}{h}\left(\begin{array}{c|c c c c}
        gh + u_m^2 + \sum\limits_{i=1}^N \frac{\alpha_i^2}{2i + 1} & -u_m & - \frac{\alpha_1}{3} & \dots & -\frac{\alpha_N}{2N + 1} \\ \hline
        -u_m & 1 & 0 & \dots & 0 \\
        -\frac{\alpha_1}{3} & 0 & \frac{1}{3} & \ddots & \vdots \\
        \vdots & \vdots & \ddots & \ddots & 0 \\
        -\frac{\alpha_N}{2N+1} & 0 & \dots & 0 & \frac{1}{2N + 1}
    \end{array}\right)
= \frac{1}{h} \left(\begin{array}{c|c} 
a & \mbf{b}^T \\ \hline
\mbf{b} & \mbf{C}
\end{array}\right).
\end{equation}
Using this decomposition and the fact that $\mbf{C}$ is diagonal, the determinants of $\mbf{H}^{-1}$ and its leading principal minors can be computed using the Schur complement formula
\begin{equation}
\det \left(\begin{array}{c|c} 
a & \mbf{b}^T \\ \hline
\mbf{b} & \mbf{C}
\end{array}\right) = \det{\left(\mbf{C}\right)}\det{\left(a - \mbf{b}^T\mbf{C}^{-1}\mbf{b}\right)}.
\end{equation}

Let $\mbf{H}^{-1}_k$ with $k=1,...,N+2$ be the leading principal minor matrices of $\mbf{H}^{-1}$. Then for $h>0$ the determinants
\begin{equation}
    \begin{aligned}
    \det\left(\mbf{H}^{-1}_1\right) &= \frac{1}{h} \left(gh + u_m^2 + \sum\limits_{i=1}^N \frac{\alpha_i^2}{2i + 1}\right) \\
    \det\left(\mbf{H}^{-1}_2\right) &= \frac{1}{h^2} \left(gh + \sum\limits_{i=1}^N \frac{\alpha_i^2}{2i + 1}\right)\\
    \det\left(\mbf{H}^{-1}_l\right) &= \frac{1}{h^l}\prod_{j=1}^{l-2} \frac{1}{2j + 1} \left(gh + \sum\limits_{i=l-1}^N \frac{\alpha_i^2}{2i + 1}\right), \quad l = 3,...,N+2,
    \end{aligned}
\end{equation}
are strictly positive and Sylvester's criterion implies that $\mbf{H}^{-1}$ is symmetric positive definite.
\end{proof}

\bibliographystyle{elsarticle-num}
\bibliography{references}
\end{document}